\documentclass{amsart}

\usepackage{amsmath, amssymb, amsthm, setspace, hyperref, xypic}
\usepackage[margin=1.25in]{geometry}
\newcommand{\norm}[1]{\left\lVert#1\right\rVert}
\newcommand{\ddb}{\partial\bar{\partial}}

\theoremstyle{definition}
\newtheorem{definition}{Definition}

\newtheorem{example}[definition]{Example}
\newtheorem{remark}[definition]{Remark}

\newtheorem{notation}[definition]{Notation}

\theoremstyle{theorem}
\newtheorem{proposition}[definition]{Proposition}
\newtheorem{theorem}[definition]{Theorem}
\newtheorem{corollary}[definition]{Corollary}
\newtheorem{lemma}[definition]{Lemma}

\title{Hermitian--Einstein metrics on stable vector bundles over compact K\"ahler orbifolds}
\author{Mitchell Faulk}
\address{Department of Mathematics, Columbia University, New York, NY}
\email{faulk@math.columbia.edu}

\begin{document}

\maketitle

\begin{abstract}
For a holomorphic vector bundle over a compact K\"ahler orbifold, the slope stability of the bundle is shown to be equivalent to the existence of a Hermitian--Einstein metric or to the properness of a certain functional introduced by Donaldson. 
\end{abstract}

\medskip

\tableofcontents

\section{Introduction}

A seminal result due to Uhlenbeck-Yau \cite{uy} states that a stable vector bundle over a compact K\"ahler manifold admits a unique Hermitian-Einstein metric. This result was also proved for surfaces by Donaldson \cite{donaldson}, and in that paper, he studied a corresponding variational problem to introduce a functional $M_K$ on the space of Hermitian metrics whose critical points are the desired Hermitian-Einstein ones. This functional was studied in a slightly more general setting by Simpson \cite{simpson}, who related the properness of this functional (in a certain sense) to the stability of the bundle in order to provide another approach to proving the result of Uhlenbeck-Yau.   

The purpose of this note is to show that it is possible to extend these results to the setting of orbifolds to obtain the following. 

\begin{theorem}\label{thm:uy}
Let $\mathcal{E}$ be an indecomposable holomorphic vector bundle over a compact K\"ahler orbifold $(\mathcal{X},\omega)$. The following statements are equivalent. 
\begin{enumerate}
\item[(i)] The bundle $\mathcal{E}$ is stable.  
\item[(ii)] For each  metric $K$ on $\mathcal{E}$, the Donaldson functional $M_K$ is proper (in the sense of Definition \ref{def:proper}).
\item[(iii)] There is a Hermitian-Einstein metric on $\mathcal{E}$. 
\end{enumerate}
\end{theorem} 

Although most of the analytical arguments extend mutatis mutandis to this orbifold setting, the algebraic ones require some care, and there are some minor points of subtlety. One of the most significant is a type of regularity result (Lemma \ref{lem:weak}), following \cite{uy}, in which a weakly holomorphic subbundle is shown to determine a coherent sheaf.   

The reader may also be interested in certain extensions or analogues of the main results of this paper.  For example, in \cite{es},  Eyssidieux and Sala provide stacky analogues of the Uhlenbeck-Yau theorem and some of its variants while also studying applications to ALE spaces. At the same time, Wenhao Ou \cite{ou2022admissible} studies the situation where the underlying space is a general compact K\"ahler variety, possibly with singularities that are worse than orbifold ones. 

\section*{Acknowledgements}

The author would like to thank Hans-Joachim Hein, Duong Phong, Sebastien Picard, Freid Tong, and Chiu-Chu Melissa Liu for helpful discussions and suggestions. The author is especially grateful to Chiu-Chu Melissa Liu for her relentness encouragement and support.  This material is based upon work supported by the National Science Foundation Graduate Research Fellowship Program under Grant No. DGE 16-44869. Any opinions, findings, and conclusions or recommendations expressed in this material are those of the author(s) and do not necessarily reflect the views of the National Science Foundation.

\section{Preliminaries}

This note follows closely the conventions and terminology about orbifolds presented in \cite{faulk1}. 

By an analytic subvariety $\mathcal{V}$ of an orbifold $\mathcal{X}$ we mean we are given the data of an atlas of charts $(U_\alpha, G_\alpha, \pi_\alpha)$ for $\mathcal{X}$ and for each $\alpha$ there corresponds a subvariety $V_\alpha$ of $U_\alpha$ satisfying $G_\alpha \cdot V_\alpha = V_\alpha$. Moreover, these $V_\alpha$ are required to agree with one another with respect to the embeddings $\lambda : U_\alpha \to U_\beta$. An analytic subvariety $\mathcal{V}$ determines a subset $V$ of the underlying space $X$ in a natural way. 

For an analytic subvariety $\mathcal{V}$ of $\mathcal{X}$, the orbifold structure on $X$ induces an orbifold structure on the complement $X \setminus V$, and we denote the resulting orbifold by $\mathcal{X} \setminus \mathcal{V}$. 

A Hermitian metric on a complex vector bundle $\mathcal{E}$ consists of a collection of hermitian metrics $H_\alpha$ on bundles $E_\alpha$ over $U_\alpha$ which are invariant under the action of $G_\alpha$ on $U_\alpha$ and which are compatible with the embeddings in the sense that for each embedding $\lambda : U_\alpha \to U_\beta$, the pullback metric $\lambda^*H_\beta$ agrees with $H_\alpha$. A Hermitian metric $H$ can be regarded as a section of the bundle $\mathcal{E}^* \otimes \overline{\mathcal{E}}^*$. A Hermitian metric is said to be compatible with a connection $D$ if $DH = 0$. 

A complex vector bundle $\mathcal{E}$ of rank $k$ over a complex $\mathcal{X}$ of dimension $n$ is called holomorphic if the transition functions $g_{\lambda} : U_\alpha \to GL(k, \mathbb{C})$ can be taken to be holomorphic. In such a case, the orbifold $\mathcal{E}$ enjoys the structure of a complex orbifold (of dimension $n + k$) in such a way that the map of orbifolds $p : \mathcal{E} \to \mathcal{X}$ is holomorphic. A holomorphic structure on $\mathcal{E}$ determines a raising operator $\bar{\partial} : A^0(\mathcal{E}) \to A^{0,1}(\mathcal{E})$ by the usual local definition, and $\bar{\partial}$ satisfies the property that if $s$ is a holomorphic section of $\mathcal{E}$, then $\bar{\partial}s = 0$. We say that a connection $D$ on $\mathcal{E}$ is compatible with the holomorphic structure if $D'' = \bar{\partial}$, where $D''$ denotes the composition of $D$ with the projection of $A^1(\mathcal{E})$ onto $A^{0,1}(\mathcal{E})$.

\begin{example}
The complexified tangent bundle $T\mathcal{X}$ of a complex orbifold $\mathcal{X}$ is a holomorphic vector bundle in the same way that it is for manifolds. 
 \end{example}

Just as in the manifold setting, if $\mathcal{E}$ is a holomorphic vector bundle, then a Hermitian metric $H$ on $\mathcal{E}$ determines a unique Chern connection, denoted $d_H$, which is compatible with $H$ and which is compatible with the holomorphic structure. The curvature $F_H$ of $d_H$ is an $\text{End}(\mathcal{E})$-valued $(1,1)$-form. 

If two metrics $H,K$ on $\mathcal{E}$ satisfy
\[
\langle \xi, \eta \rangle_H = \langle h\xi,\eta \rangle_K
\]
for a positive endomorphism $h$ of $\mathcal{E}$, we write $H = Kh$. One can show that in such a case, the endomorphism $h$ is self-adjoint with respect to $K$ (and also $H$). In addition, the curvatures $F_H$ and $F_K$ are related by 
\[
F_H = F_K + \bar{\partial}(h^{-1} \partial_{K}h),
\]
where we are using the notation $\partial_K$ to denote the $(1,0)$-component of the Chern connection $d_K$. 

\begin{lemma}\label{lem:curvaturerelations}
If two metrics $H,K$ satisfy $H = Ke^s$ for an endomorphism $s$ that is self-adjoint with respect to $K$, then
\begin{enumerate}
\item[(i)] the adjoint of $\partial_K s$ is $\bar{\partial}s$.
\item[(ii)] $\Delta_{\partial_K} s = i \Lambda (F_H - F_K)$
\item[(iii)] $\Delta_{\bar{\partial}}s = i\Lambda(F_H - F_K) - i\Lambda F_Ks$
\item[(iv)] $\norm{{\partial}_Ks}_{L^2_K}^2 = \langle i \Lambda (F_H - F_K), s \rangle_{L^2_K}$
\item[(v)] $\Delta |s|_K^2 = \langle 2i \Lambda (F_H - F_K), s \rangle - \langle i \Lambda F_K s, s \rangle - |d_Ks|^2$
\end{enumerate} 
\end{lemma}

\begin{proof}
For (i), upon differentiating the relation $\langle s \eta, \xi \rangle = \langle \eta, s \xi \rangle$, we find 
\[
\langle (\partial_K s) \eta + s(\partial_K \eta), \xi \rangle + \langle s \eta, \bar{\partial} \xi \rangle = \langle \partial_K \eta, s \xi \rangle + \langle \eta, (\bar{\partial}s) \xi + s \bar{\partial \xi} \rangle.
\]
Because $s$ is self-adjoint we are left with 
\[
\langle (\partial_K s)\eta, \xi \rangle = \langle \eta, (\bar{\partial}s) \xi \rangle.
\]

For (ii), the curvatures are related by 
\[
F_H = F_K + \bar{\partial}(H^{-1} \partial_K H) = F_K + \bar{\partial}( \partial_K s).
\]
The K\"ahler identities (see \cite{gh}) extend to identities on bundle-valued forms  to imply the relation
\[
\partial_K^* = i \Lambda \bar{\partial},
\]
which gives that 
\[
i\Lambda F_H = i\Lambda F_K + \partial^*_K \partial_Ks.
\]
We conclude that 
\[
\Delta_{\partial_K} s = i \Lambda (F_H - F_K),
\]
as claimed.

For (iii), we recall that $F_K$ is given by $F_K = \partial_K \bar{\partial} + \bar{\partial}\partial_K$ so that the curvatures are related by 
\[
F_H = F_K + (F_K - \partial_K \bar{\partial})s.
\]
We then use the K\"ahler identity $\bar{\partial}^* = -i \Lambda \partial_K$ to obtain that 
\[
i\Lambda F_H = i\Lambda F_K + i\Lambda F_K s + \Delta_{\bar{\partial}}s.
\]
Rearranging gives (iii). 

For (iv), multiplying the equality of (ii) on the right by $s$, then taking the trace, and then integrating gives 
\begin{align*}
\int_{\mathcal{X}} |\partial_K s|_K^2 \frac{\omega^n}{n!} &= \int_{\mathcal{X}} \langle \partial_K^* \partial_K s, s \rangle_K \frac{\omega^n}{n!} \\
&= \int_{\mathcal{X}} \langle i \Lambda (F_H - F_K), s \rangle_K \frac{\omega^n}{n!},
\end{align*} 
as desired. 

For (v), we have the identity 
\[
\Delta |s|_K^2 =  \frac{1}{2} \Delta_d |s|_K^2 =  \langle \Delta_K s, s \rangle -  |d_K s|^2
\]
regardless of whether $s$ is self-adjoint. (Here $\Delta_d$ denotes the de Rham Laplacian.) Because $\Delta_K = \Delta_{\partial_K} + \Delta_{\bar{\partial}}$, we obtain from parts (ii) and (iii) that 
\[
\Delta |s|_K^2 = \langle 2i \Lambda (F_H - F_K), s \rangle - \langle i \Lambda F_K s, s \rangle - |d_ks|^2
\] 
as desired. 
\end{proof}

\subsection{Stable bundles and sheaves}

Chern-Weil theory may be used as usual to define Chern classes (or more generally characteristic classes) of vector bundles. We will discuss characteristic classes in greater detail in Section \ref{sec:donaldson}, but for now let us at least note that the first Chern class $c_1(\mathcal{E})$ can be defined as the cohomology class represented by the $(1,1)$-form
\[
\frac{i}{2\pi} \text{Tr}(F_H)
\]
for any choice of Hermitian metric $H$ on $\mathcal{E}$. The degree of $\mathcal{E}$ is then the integral of $\frac{i}{2\pi}\Lambda \text{Tr}(F_H)$ over $\mathcal{X}$
\[
\deg(\mathcal{E}) = \frac{i}{2\pi} \int_{\mathcal{X}} \Lambda \text{Tr}(F_H) \cdot \text{vol} = \int_{\mathcal{X}} c_1(\mathcal{E}) \wedge \frac{\omega^{n-1}}{(n-1)!},
\]
and the slope of $\mathcal{E}$ is the ratio 
\[
\mu(\mathcal{E}) = \frac{\text{deg}(\mathcal{E})}{\text{rank}(\mathcal{E})}.
\]

Given an action $\sigma : G \times U \to U$ of a finite group $G$ on $U \subset \mathbb{C}^n$ by biholomorphisms, a $G$-equivariant sheaf over $U$ consists of the data of a sheaf $\mathcal{F}$ of $\mathcal{O}_U$-modules together with an isomorphism of sheaves of $\mathcal{O}_{G \times U}$-modules
\[
\rho : \sigma^* \mathcal{F} \to p_2^*\mathcal{F}
\]
which satisfies the cocycle relation 
\[
p_{23}^*\rho \circ (1_G \times \sigma)^* \rho = (m \times 1_U)^*\rho
\]
where $m$ denotes multiplication $m : G \times G \to G$ and $p_{23} : G \times G \times U \to G \times U$ is the projection onto the second two factors. 

By a sheaf $\mathcal{F}$ over $\mathcal{X}$ we mean we are given the data of an atlas $(U_\alpha, G_\alpha, \pi_\alpha)$ of orbifold charts together with a $G_\alpha$-equivariant sheaf $\mathcal{F}_\alpha$ over each $U_\alpha$. For each embedding $\lambda : U_\alpha \to U_\beta$ there also corresponds a sheaf isomorphism $\tau_{\lambda} : \mathcal{F}_\alpha \to \lambda^*\mathcal{F}_\beta$. Moreover these isomorphisms are compatible with one another in the sense that whenever $\lambda : U_\alpha \to U_\beta$ and $\lambda' : U_\beta \to U_\gamma$ are a pair of composable embeddings, then $\tau_{\lambda' \circ \lambda} = \lambda^*\tau_\lambda' \circ \tau_\lambda.$

The notion of a sheaf $\mathcal{F}$ over an analytic subvariety $\mathcal{V}$ of $\mathcal{X}$ is defined similarly. In particular, if $\mathcal{V}$ is given locally by subvarieties $V_\alpha$ of charts $(U_\alpha, G_\alpha, \pi_\alpha)$, then a sheaf assigns to each $V_\alpha$ a $G_\alpha$-equivariant sheaf $\mathcal{F}_\alpha$,  and moreover to each embedding of charts, there corresponds a sheaf isomorphism as above (and these isomorphisms are compatible with one another). It is important to note that a sheaf $\mathcal{F}$ over $\mathcal{X}$ is not the same thing as a sheaf over the underlying topological space $X$. 

\begin{example}
A complex orbifold $\mathcal{X}$ enjoys a structure sheaf $\mathcal{O}_{\mathcal{X}}$ of holomorphic $\mathbb{C}$-valued functions, and more generally, any analytic subvariety $\mathcal{V}$ of $\mathcal{X}$ determines a structure sheaf $\mathcal{O}_{\mathcal{V}}$. 
\end{example}

A sheaf $\mathcal{F}$ is called coherent (resp. torsion-free) if each $F_\alpha$ is. If $\mathcal{F}$ is coherent and torsion-free of rank $r$, then one can define the determinant line bundle associated to $\mathcal{F}$ to be 
\[
\det(\mathcal{F}) = (\Lambda^r \mathcal{F})^{**}.
\]
It follows that for a coherent torsion-free sheaf $\mathcal{F}$ we have a well-defined notion of degree 
\[
\deg(\mathcal{F}) = \deg(\det(\mathcal{F}))
\]
and slope
\[
\mu(\mathcal{F}) = \frac{\deg(\mathcal{F})}{\text{rank}(\mathcal{F})}.
\]

In addition, one can show that a torsion-free coherent sheaf is locally free outside of a subset of codimension at least two. A proof of this can be found for example in \cite{kobayashi}. The argument given there extends to the setting of orbifolds because one can apply the argument to each $G_\alpha$-equivarariant sheaf over each chart.

\begin{lemma}\label{lem:subsheaf}
If $\mathcal{F}$ is a torsion-free coherent sheaf, then there is a subvariety $\mathcal{V}$ of codimension at least $2$ in $X$ such that the restriction of $\mathcal{F}$ to $\mathcal{X} \setminus \mathcal{V}$ is locally free. 
\end{lemma}

The notion of slope allows one to introduce the usual notion of (slope) stability in the standard way.

\begin{definition}
One says that a coherent torsion-free sheaf $\mathcal{F}$ is semi-stable if for each proper coherent subsheaf $\mathcal{F}'$ of $\mathcal{F}$, we have the inequality $\mu(\mathcal{F}') \leqslant \mu(\mathcal{F})$. If moreover the strict inequality $\mu(\mathcal{F}') < \mu(\mathcal{F})$ holds for each proper coherent subset $\mathcal{F}'$ satisfying $0 < \text{rank}(\mathcal{F}') < \text{rank}(\mathcal{F})$, then we say that $\mathcal{F}$ is stable. In addition, a holomorphic vector bundle $\mathcal{E}$ is called (semi)-stable if its corresponding sheaf of sections is. 
\end{definition}

\subsection{The heat flow and Donaldson functional}\label{sec:donaldson}

Let us fix from this point forward a holomorphic vector bundle $\mathcal{E}$ of rank $r$ over a K\"ahler orbifold $(\mathcal{X}, \omega)$. There will be no loss of generality in assuming in addition that $\mathcal{E}$ is indecomposable. 

\begin{definition}
A Hermitian metric $H$ on $\mathcal{E}$ is called Hermitian-Einstein if there is a constant $\lambda$ such that 
\begin{align}\label{eqn:HE}
\Lambda F_H = \lambda \cdot I_\mathcal{E} \in A^0(\textnormal{End}(\mathcal{E}))
\end{align}
where $I_\mathcal{E}$ denotes the identity automorphism on $\mathcal{E}$. 
\end{definition}

\begin{remark}\label{rem:HE}
The constant $\lambda = \lambda(\mathcal{X}, \omega, \mathcal{E})$ can be determined by the K\"ahler class $[\omega]$ and the slope of $\mathcal{E}$. In particular, taking the trace of both sides of \eqref{eqn:HE} and integrating over $\mathcal{X}$ gives 
\[
\deg(\mathcal{E}) = \frac{\lambda i}{2\pi} \text{rank}(\mathcal{E}) \text{vol}(\mathcal{X})
\]
so that 
\[
\lambda = \frac{-2\pi i \cdot \mu(E)}{\text{vol}(\mathcal{X})}. 
\]
\end{remark}

\begin{definition}\label{defn:heatflow}
By a heat flow with initial data $H_0$ we mean a flow of metrics $H_t$ satisfying the differential equation
\begin{align}\label{eqn:heat}
\dot{H}_t = -\frac{i}{2} H_t(\Lambda F_{t} - \lambda \cdot I_{\mathcal{E}}). 
\end{align} 
\end{definition}

In particular, note that stable points of this flow must be Hermitian-Einstein metrics. Donaldson studied this flow in \cite{donaldson}, and some of the results from that paper can be summarized in the following theorem. 

\begin{theorem}
For any initial metric $H_0$ on $\mathcal{E}$, the heat flow \eqref{eqn:heat} has a unique smooth solution defined for $0 \leqslant t < \infty$. 
\end{theorem}

This theorem is valid for orbifolds for a few reasons. First, the short-time existence is a local argument involving a linearization of the flow, which can be studied in a local orbifold chart with no changes from the manifold setting. The long-time existence involves estimates to solutions of the flow, which remain valid in the orbifold setting since in particular the estimates are valid on each chart in some open cover, whereby computations agree with those in the manifold case. Some of the intermediary results Donaldson obtained in order to establish long-time existence included the following two propositions.  

\begin{proposition}\label{prop:decreasing}
For an initial metric $H_0 = K$ on $\mathcal{E}$, the function 
\[
\sup_{\mathcal{X}} |\Lambda F_t - \lambda I_{\mathcal{E}}|_K^2 
\]
is decreasing along the heat flow. 
\end{proposition}

\begin{proposition}\label{prop:l2p}
For an initial metric $H_0 = K$ on $\mathcal{E}$, let $H_t$ be a one-parameter family of metrics for $0 \leqslant t < T$. Assume that $H_t$ converges in $C^0$-norm (with respect to $K$) to some continuous metric as $t \to T$ and also that we have a uniform bound on $\sup_{\mathcal{X}}|\Lambda F_t|_K^2$. Then we also have a uniform $L_2^p$-bound on $H_t$ for each $p < \infty$ (where the norm is computed with respect to $K$). Moreover, this result is still true when we allow $T = \infty$. 
\end{proposition}

As a result of this former proposition, we have the following corollary. 

\begin{corollary}\label{cor:l12}
Let $H_t$ be a solution to the heat flow with initial condition $H_0 = K$. If $H_t$ is uniformly bounded with respect to the $C^0$-norm, then $H_t$ is also uniformly bounded with respect to the $L_1^2$-norm, where here all norms are computed with respect to the initial metric $K$.
\end{corollary}

\begin{proof}
Proposition \ref{prop:decreasing} implies that $|\Lambda F_t|_K$ is uniformly bounded with respect to the $C^0$-norm.  In addition, we are assuming a $C^0$-bound on $H_t$. Thus the right-hand side of equality (iv) of Lemma \ref{lem:curvaturerelations} is bounded uniformly by a constant independent of $t$. The result then follows.  
\end{proof}

Another ingredient found in Donaldson \cite{donaldson} is the following result concerning the $C^0$-norm of solutions to the heat flow. For two metrics $H_1, H_2$, let $\sigma(H_1,H_2)$ denote the number 
\[
\sigma(H_1, H_2) = \text{Tr}(H_1^{-1}H_2) + \text{Tr}(H_2^{-1}H_1) - 2 r,
\]
where $r$ is the rank of $\mathcal{E}$. Then the assignment $\sigma$ does not quite define a metric, but we do have in fact that a sequence $H_i$ converges to  $H$ in $C^0$ if and only if $\sup_{\mathcal{X}} \sigma(H_i, H) \to 0$. (In fact, the space of hermitian metrics is the set of sections of a fiber bundle, which, on each fiber admits a natural distance function $d$ coming from the description of the fiber as a homogeneous space $GL(r,\mathbb{C})/U(r)$, and Donaldson \cite{donaldson} asserts that the function $\sigma$ compares uniformly with $d$, in the sense that $\sigma \leqslant f(d)$ and $d \leqslant F(\sigma)$ for monotone $f,F$.) Donaldson \cite{donaldson} then proves the following by a direct calculation. 

\begin{proposition}
If $H_t, K_t$ are two solutions to the heat flow and $\sigma = \sigma(H_t, K_t)$, then 
\[
\left(\frac{\partial}{\partial t} + \Delta \right)\sigma \leqslant 0.
\]
\end{proposition}

With this, it is possible to obtain the following result which relates the $L^2$-convergence of solutions to the heat flow with $C^0$-convergence.

\begin{corollary}\label{cor:l2toc0}
For a real number $\tau > 0$ and a solution $H_t$ to the heat flow, let $\sigma^\tau = \sigma(H_t, H_{t + \tau})$. Then for each $t' > 0$, we have 
\[
\sup_{\mathcal{X}} \sigma^\tau(t + t') \leqslant c(t')\int_{\mathcal{X}} \sigma^\tau(t) \frac{\omega^n}{n!},
\]
where $c(t')$ denotes the (finite) supremum of the heat kernel at time $t'$. In particular, if the sequence $H_t$ coverges in $L^2$ as $t \to \infty$, then it also converges in $C^0$, where here the norms are to be computed with respect to the fixed initial metric $K = H_0$.  
\end{corollary}

\begin{proof}
The inequality follows immediately from the previous proposition together with a type of Green's formula involving the heat kernel (see \cite{eells}). For the second part about convergence, suppose that $H_t$ is Cauchy in $L^2$. Let $\epsilon > 0$ be given. Because $H_t$ is Cauchy in $L^2$ and because the function $\sigma$ compares uniformly with the norm afforded by $K$, there is a time $T > 0$ such that 
\[
\int_{\mathcal{X}} \sigma^\tau(t)\frac{\omega^n}{n!} < c(1)^{-1} \epsilon
\]
for each $t > T$ and each $\tau > 0$. We therefore find that 
\[
\sup_{\mathcal{X}} \sigma^\tau(t+1) \leqslant c(1)\int_{\mathcal{X}} \sigma^\tau(t) \frac{\omega^n}{n!} < \epsilon
\]
for each $t > T$ and each $\tau > 0$. This means precisely that 
\[
\sup_{\mathcal{X}} \sigma(H_{t+1}, H_{t + 1 + \tau}) < \epsilon
\] 
for each $t > T$ and each $\tau > 0$. We conclude that the $H_t$ are uniformly Cauchy with respect to the $C^0$-norm determined by $K$. 
\end{proof}

Also in \cite{donaldson}, Donaldson considered a corresponding variational approach and introduced a functional whose critical points correspond to stable points of the heat flow. Such a functional can be defined using the notion of secondary characteristic classes, which we review now. 

A $p$-multilinear function $\varphi$ on $\mathfrak{gl}(r, \mathbb{C})$ is called invariant if it is invariant under the (diagonal) adjoint action of $GL(r, \mathbb{C})$ on $p$ copies of $\mathfrak{gl}(r, \mathbb{C})$. Such a function assigns to any metric $H$ on $\mathcal{E}$ a $(p,p)$-form
\[
\varphi(F_H) := \varphi(F_H, \ldots, F_H) \in A^{p,p}(\mathcal{X}).
\]
The cohomology class represented by the $(p,p)$-form $\varphi(F_H)$ is independent of the choice of metric $H$ and is called the first characteristic class associated to $\varphi$. In addition the $\ddb$-lemma implies that the difference between two such forms is $\ddb$-exact. In fact, the following more precise statement is true (see \cite{donaldson}). 

\begin{proposition}
If $H,K$ are two metrics on $\mathcal{E}$, then for each $\varphi$ there is an invariant 
\[
R_\varphi(H,K) \in A^{p-1, p-1}(\mathcal{X})/(\textnormal{Im} \partial + \textnormal{Im} \bar{\partial}),
\]
called the secondary characteristic class associated to $\varphi$,
satisfying the following three properties. 
\begin{enumerate}
\item[(i)] We have $R_\varphi(K,K) = 0$ and for any third metric $J$, we have 
\[
R_\varphi(H,K) = R_\varphi(H,J) + R_\varphi(J,K).
\]
\item[(ii)] If $H_t$ is a smooth family of metrics, then 
\[
\frac{d}{dt}R_\varphi(H_t, K) = -i \varphi(F_{H_t};H^{-1}_t \dot{H}),
\]
where $\varphi(F_{H_t};H^{-1}_t \dot{H})$ denotes the sum 
\[
\varphi(F_{H_t};H^{-1}_t \dot{H}) = \sum_{k=1}^{p} \varphi(F_{H_t}, \ldots, \overbrace{H^{-1}_t \dot{H}}^{k}, \ldots, F_{H_t}).
\]
\item[(iii)] We have
\[
i \bar{\partial}\partial R_\varphi(H,K) = \varphi(F_H) - \varphi(F_K) \in A^{p,p}(X).
\]
\end{enumerate}
\end{proposition}

\begin{example}
For our purposes, we only need two such invariants $R_\varphi$. The first $R_1$ is associated to the trace $\varphi_1(A) = \text{Tr}A$, and the second $R_2$ is associated to the Killing form $\varphi_2(A,B) = -\text{Tr}(AB)$.

This means in particular that given a path $H_t$ of metrics with $H_0 = K$ and $H_1 = H$, we may set 
\begin{align*}
R_1(H,K) &= -i \int_0^1 \text{Tr}(H_t^{-1}\dot{H}) \; dt \\
R_2(H,K) &= 2i \int_0^1 \text{Tr}(H_t^{-1}\dot{H} F_t) \; dt.
\end{align*}
The particular integrals may depend on the choice of path, but, modulo $\text{Im} \partial + \text{Im} \bar{\partial}$, they do not. 
\end{example}

\begin{definition}
With these two invariants, Donaldson introduced a functional for surfaces, whose extension to arbitrary dimensions can be described as 
\begin{align}  
M(H,K) = \int_{\mathcal{X}} (R_2 + 2 \lambda R_1  \omega) \wedge \frac{\omega^{n-1}}{(n-1)!}.
\end{align}
For a fixed metric $K$, we can consider the functional $M(-,K)$ on the space of metrics, and it turns out that the critical points of this functional (if they exist) are Hermitian-Einstein metrics. 
\end{definition}

\begin{proposition}\label{prop:M_K}
For a fixed metric $K$, let $M_K$ denote the functional on the space of metrics described by $M_K(H) = M(H,K)$. 
\begin{enumerate}
\item[(i)] If $H_t$ is any smooth path of metrics, then the variation of $M_K(H_t)$ along $H_t$ is given by  
\[
\frac{\partial}{\partial t}M_K(H_t) = 2i \int_{\mathcal{X}} \textnormal {Tr}(H_t^{-1} \dot{H}_t(F_t - \lambda \omega I_{\mathcal{E}})) \wedge \frac{\omega^{n-1}}{(n-1)!}.
\]
\item[(ii)] If $H$ is a critical point of $M_K$, then $H$ is a Hermitian-Einstein metric. 
\item[(iii)] In particular, if $H_t$ is a solution to the heat flow \eqref{eqn:heat}, then 
\[
\frac{\partial}{\partial t}M_K(H_t) = - \norm{\Lambda F_t - \lambda I_\mathcal{E}}^2_{L^2_{H_t}},
\] 
meaning that $M_K$ is a non-increasing function of $t$ along the heat flow. 
\end{enumerate}
\end{proposition}

\begin{proof}
For part (i), we compute using the definitions of the invariants $R_1$ and $R_2$ as follows 
\begin{align*}
\frac{\partial}{\partial t}M_K(H_t) &= \int_{\mathcal{X}} \left( 2i \text{Tr}(H_t^{-1} \dot{H} F_t)  - 2i\lambda  \text{Tr}(H_t^{-1} \dot{H}) \omega\right) \wedge \frac{\omega^{n-1}}{(n-1)!} \\
&= 2i \int_{\mathcal{X}} \text{Tr}(H_t^{-1} \dot{H}(F_t - \lambda\omega I_{\mathcal{E}})) \wedge \frac{\omega^{n-1}}{(n-1)!}
\end{align*}
Part (ii) is then immediate from the computation in part (i). Part (iii) then follows from using the flow \eqref{eqn:heat} and the fact that $\Lambda F_t - \lambda I_{\mathcal{E}}$ is skew-adjoint with respect to $H_t$. 
\end{proof}

Part (iii) of the previous proposition says that $M_K$ is non-increasing along the heat flow, and we will later show that the functional $M_K$ is convex in a certain sense (see Proposition \ref{prop:M_Kconvex}). In general, however, $M_K$ may not be bounded from below. Because the critical points of $M_K$ are the desired metrics, it would be useful to understand exactly when $M_K$ admits such critical points. Motivated by Proposition 5.3 of \cite{simpson}, we introduce the following notion of properness for $M_K$. 

\begin{definition}\label{def:proper}
We say that $M_K$ is proper if there are positive constants $C_1, C_2$ such that for the solution $H_t = Ke^{s_t}$ to the heat flow with initial condition $s_0 = 0$, we have 
\[
\sup_{\mathcal{X}}|s_t|_K \leqslant C_1 + C_2 M_K(Ke^{s_t})
\]
for each $t \geqslant 0$. 
\end{definition}

\begin{corollary}\label{cor:proper}
If $M_K$ is proper and $H_t$ is a solution to the heat flow with initial condition $H_0 = K$, then the following statements are true. 
\begin{enumerate}
\item[(i)] $M_K(H_t)$ is bounded from below (by $-C_1/C_2$). 
\item[(ii)] $\norm{H_t}_{C^0_K}$ is bounded from above. 
\item[(iii)] $\Lambda F_{H_t} \to \lambda I_{\mathcal{E}}$ in $L^2_K$ as $t \to \infty$. 
\end{enumerate}
\end{corollary}

\begin{proof}
Part (i) is obvious. Part (ii) follows from the fact that $t \mapsto M_K(H_t)$ is decreasing along the heat flow and $M_K(H_0) = M_K(K) = 0.$ For part (iii), because $M_K(H_t)$ is bounded from below and non-increasing, we know that 
\[
\lim_{t \to \infty} \frac{\partial}{\partial t} M_K (H_t) = 0 
\] 
and by the previous proposition we conclude that 
\[
\lim_{t \to \infty} \norm{\Lambda F_t - \lambda I_{\mathcal{E}}}_{L^2_{H_t}}^2 = 0.
\]
Because $H_t$ is uniformly bounded in $C^0$, the $L^2$-norm with respect to $H_t$ is equivalent to the $L^2$-norm with respect to $K$ in a uniform manner, meaning that there is a constant $C$ independent of $t$ such that 
\[
\norm{\Lambda F_t - \lambda I_{\mathcal{E}}}_{L^2_{K}}^2 \leqslant C \norm{\Lambda F_t - \lambda I_{\mathcal{E}}}_{L^2_{H_t}}^2.
\]
Taking the limit of both sides gives the required convergence.  
\end{proof}

\section{Proof of the main result}

Let us recall that our objective is to prove the following. 

\begin{theorem}\label{thm:main}
Assume $\mathcal{E}$ is indecomposable. Then the following are equivalent. 
\begin{enumerate}
\item[(i)] The bundle $\mathcal{E}$ is stable. 
\item[(ii)] For each metric $K$ on $\mathcal{E}$, the Donaldson functional $M_K$ is proper in the sense of Definition \ref{def:proper}. 
\item[(iii)] There is a Hermitian-Einstein metric on $\mathcal{E}$. 
\end{enumerate}
\end{theorem}

Let us immediately deal with the implication (iii) $\implies$ (i). A proof can be found, for example, in \cite{lubke}, but we outline a proof now for the sake of completeness.

\begin{proposition}\label{prop:iii}
Assume $\mathcal{E}$ is indecomposable. If there is a Hermitian-Einstein metric $H$ on $\mathcal{E}$, then $\mathcal{E}$ is stable.  
\end{proposition}

\begin{proof}
Let $\mathcal{E}'$ be a proper coherent subsheaf of $\mathcal{E}$ of rank $r'$ with torsion-free quotient $\mathcal{E}/\mathcal{E}'$. Lemma \ref{lem:subsheaf} implies that $\mathcal{E}'$ is locally free outside of a subvariety of codimension at least $2$. From this point forward, we work away from this subvariety so that for example the metric $H$ restricts to a metric on $\mathcal{E}'$ with corresponding curvature denoted $F'$. It is standard (see \cite[Chapter 1, Section 5]{gh}) to show that the difference of curvatures 
\[
F|_{\mathcal{E}'} - F' 
\]
is a semi-positive $\text{End}(\mathcal{E}')$-valued $(1,1)$-form and moreover vanishes if and only if the orthogonal complement of $\mathcal{E}'$ is holomorphic. Here we are using the convention as in \cite{gh} that semi-positive implies in particular that 
\[
\frac{i}{2\pi} \cdot \text{Tr}_{\mathcal{E}'}(F)  - \frac{i}{2\pi} \text{Tr}_{\mathcal{E}'}F'
\]
is a positive $(1,1)$-form.   Integrating over $\mathcal{X}$ we obtain the inequality 
\begin{align}\label{eqn:int1}
 \frac{i}{2\pi} \int_{\mathcal{X}} \text{Tr}_{\mathcal{E}'}\Lambda F \cdot \text{vol} \geqslant \deg(\mathcal{E}') 
\end{align}
which is valid because we are working outside of a subset of codimension at least two. 
Now the Hermitian-Einstein condition guarantees that 
\[
\text{Tr}_{\mathcal{E}}(\Lambda F) = \text{Tr}_{\mathcal{E}} (\lambda I_{\mathcal{E}}) = r \cdot \lambda
\]
and hence also that 
\[
\text{Tr}_{\mathcal{E}'} \Lambda F = r' \cdot \lambda = \frac{r'}{r} \text{Tr}_{\mathcal{E}}\Lambda F. 
\]
Using these we find that \eqref{eqn:int1} is equivalent to  
\[
\frac{r'}{r} \deg(\mathcal{E}) \geqslant \deg(\mathcal{E}').
\]
But the inequality is actually strict because equality would mean that the complement of $\mathcal{E}'$ is holomorphic, which is a contradiction to the assumption that $\mathcal{E}$ is indecomposable. We conclude that $\mathcal{E}$ is stable. 
\end{proof}

A proof of the implication (ii) $\implies$ (iii) for manifolds can be found in \cite{simpson}, and for this, one uses the heat flow of Definition \ref{defn:heatflow}. Indeed the argument roughly proceeds as follows. The assumption (ii) guarantees the functional $M_K$ has a unique critical point belonging to a certain Sobolev space. The heat flow approaches this critical point $H_\infty$ and certain estimates involving this flow and the functional $M_K$ allow one to obtain enough regularity on this critical point to ascertain that $H_\infty$ corresponds to a bona fide smooth metric. Since stable points of the heat flow are Hermitian-Einstein, we find that $H_\infty$ is. Let us now be more precise. 

\begin{proposition}\label{prop:HE}
Assume $\mathcal{E}$ is indecomposable. If for each fixed metric $K$ on $\mathcal{E}$, the Donaldson functional $M_K$ is proper in the sense of Definition \ref{def:proper}, then there is a Hermitian-Einstein metric on $\mathcal{E}$. 
\end{proposition}

\begin{proof}
Let us denote by $H_t$ a solution to the heat flow with initial condition $H_0 = K$. Corollary \ref{cor:proper} applies so in particular we have a uniform $C^0$-bound on $H_t$. (Here all norms will be computed with respect to the fixed initial metric $K$.) We have a uniform $C^0$-bound on $\Lambda F_t$ by Proposition \ref{prop:decreasing}. It follows from Corollary \ref{cor:l12} that we have a uniform $L_1^2$-bound on $H_t$. A compactness theorem now guarantees the existence of a sequence of times $t_i \to \infty$ and a limit $H_\infty \in L_1^2$ such that the sequence $H_{t_i}$ converges in $L^2$ to $H_\infty$. Corollary \ref{cor:l2toc0} implies actually that $H_{t_i}$ converges to $H_\infty$ in $C^0$-norm.  Proposition \ref{prop:l2p} now gives that $H_{t_i}$ is in fact uniformly bounded in $L_2^p$ for each $p < \infty$. It follows that the weak limit $F_\infty$ exists in $L^p$ for each $p < \infty$ and moreover that the weak equation $\Lambda F_\infty = \lambda \cdot I_{\mathcal{E}}$ holds by Corollary \ref{cor:proper} (iii). Elliptic regularity now implies that $H_\infty$ is in fact smooth.
\end{proof}

The remaining implication is (i) $\implies$ (ii), and this is proved for manifolds in \cite{simpson} with the help of a regularity statement concerning weakly holomorphic subbundles from \cite{uy}, another proof of which can be found in \cite{popovici}. The precise notion of a weakly holomorphic subbundle that we will use is the following. 

\begin{definition}\label{def:weak}
By a weakly holomorphic subbundle of $\mathcal{E}$ (with respect to a fixed metric $K$) we mean an $L_1^2$ section $\Pi$ of $\text{End}(\mathcal{E})$ which satisfies $\Pi = \Pi^* = \Pi^2$ (where the adjoint is computed with respect to $K$) and $(I_{\mathcal{E}} - \Pi) \bar{\partial} \Pi = 0$. 

A weakly holomorphic subbundle determines a degree via Chern-Weil theory, which may be computed as 
\begin{align}\label{eqn:degree}
\deg(\Pi) = \frac{i}{2\pi} \int_{\mathcal{X}} \text{Tr}(\Pi \Lambda F_K) \frac{\omega^n}{n!} - \frac{1}{2\pi} \int_{\mathcal{X}} |\bar{\partial} \Pi|_K^2 \frac{\omega^n}{n!}
\end{align}
(compare to \cite[Lemma 3.2]{simpson} or \cite[Proposition 4.2]{uy}). It therefore makes sense to say when a weakly holomorphic subbundle is destabilizing for $\mathcal{E}$. 
\end{definition}

\begin{remark}
Let us verify equation \eqref{eqn:degree} for the case of a smooth subbundle $\mathcal{S}$ of $\mathcal{E}$. Let $\Pi$ denote the projection endomorphism of $\mathcal{E}$ corresponding to $\mathcal{S}$. If we write $D_{\mathcal{E}}$ for the Chern connection on $\mathcal{E}$ determined by the metric $K$, then there is a connection $D_{\mathcal{S}}$ on $\mathcal{S}$ described by the composition $\Pi \circ D_{\mathcal{E}}$. The difference $A = D_{\mathcal{E}}|_{\mathcal{S}} - D_{\mathcal{S}}$ may be considered as a map from $A^0(\mathcal{S})$ to $A^1(\mathcal{S}^\perp)$. In fact, $A$ is a map to $A^{1,0}(\mathcal{S}^\perp)$ (see \cite{gh}) and corresponds to the composition $\Pi^\perp \circ \partial_{\mathcal{E}} \Pi$ (see \cite[Proposition 4.2]{uy}), where here $\partial_{\mathcal{E}}$ denotes the $(1,0)$-component of $D_{\mathcal{E}}$. In \cite{gh}, the curvature of the subbundle is related to the curvature of the ambient bundle by $F_{\mathcal{S}} = \Pi \circ F_{\mathcal{E}} - A \wedge A^*.$ Taking the trace and integrating over $\mathcal{X}$ we find that 
\begin{align*}
\deg(\mathcal{S}) &= \frac{i}{2\pi} \int_{\mathcal{X}} \text{Tr}(F_{\mathcal{S}}) \wedge \frac{\omega^{n-1}}{(n-1)!} \\
&=  \frac{i}{2\pi} \int_{\mathcal{X}} \text{Tr}(\Pi F_{\mathcal{E}}) \frac{\omega^{n-1}}{(n-1)!} - \frac{1}{2\pi}\int_{\mathcal{X}} \text{Tr}(i A \wedge A^*) \wedge \frac{\omega^{n-1}}{(n-1)!} \\
&= \frac{i}{2\pi} \int_{\mathcal{X}} \text{Tr}(\Pi \Lambda F_{\mathcal{E}}) \frac{\omega^n}{n!} - \frac{1}{2\pi}\int_{\mathcal{X}} |\Pi^\perp \circ \partial_{\mathcal{E}} \Pi|_K^2 \frac{\omega^{n}}{n!}.
\end{align*}
Because $\Pi  \circ \partial_{\mathcal{E}} \Pi = 0$, we conclude that in fact $\Pi^\perp \circ \partial_{\mathcal{E}}\Pi = \partial_{\mathcal{E}} \Pi$. The fact that $\Pi$ is self-adjoint also implies that $|\partial_{\mathcal{E}}\Pi| = |\bar{\partial} \Pi|$ and the formula \eqref{eqn:degree} is verified.  
\end{remark}

With these notions, the implication (i) $\implies$ (ii) then follows immediately from the following two lemmas. 

\begin{lemma}\label{lem:proper}
Suppose $M_K$ is not proper. Then there is a weakly holomorphic subbundle of $\mathcal{E}$ which is destabilizing for $\mathcal{E}$. 
\end{lemma}

\begin{lemma}\label{lem:weak}
Let $\Pi$ be a weakly holomorphic subbundle of $\mathcal{E}$. Then there is a coherent subsheaf $\mathcal{F}$ of $\mathcal{E}$ and an analytic subvariety $\mathcal{V}$ of codimension at least two in $\mathcal{X}$ such that 
\begin{enumerate}
\item[(i)] The map $\Pi$ is smooth away from $\mathcal{V}$ and there we have $\Pi = \Pi^* = \Pi^2$ and $(I_{\mathcal{E}} - \Pi) \circ \bar{\partial} \Pi = 0$. 
\item[(ii)] Outside of $\mathcal{V}$ the subsheaf $\mathcal{F}$ agrees with the image of $\Pi$ and is a holomorphic subbundle of $\mathcal{E}|_{\mathcal{X}\setminus \mathcal{V}}$. 
\end{enumerate}
\end{lemma}

Lemma \ref{lem:proper} is proved for manifolds in Proposition 5.3 of \cite{simpson}, and exactly the same method of proof applies in our setting. We reserve the final section following this one for a discussion of this method. The basic idea is the following. Assuming $M_K$ is not proper, we can obtain a sequence $s_k$ of sections of $\text{End}(\mathcal{E})$ with larger and larger norms. An appropriate normed sequence $u_k$ then tends to a weak limit $u_\infty$ in $L^2_1$, whose eigenvalues are constant almost everywhere. The eigenspaces of $u_\infty$ then give rise to a filtration of $\mathcal{E}$ by $L_1^2$-subbundles, for which, it is possible to show that one must be destabilizing.  

Assuming Lemma \ref{lem:proper} then, for now,  it remains only to discuss Lemma \ref{lem:weak}. A version of this result can be found in the original paper by Uhlenbeck-Yau \cite{uy}, and we aim to explain how it extends to the orbifold setting. 

For the holomorphic vector bundle $\mathcal{E}$ of rank $r$ over $\mathcal{X}$, it is possible to form the Grassmannian bundle $\text{Gr}(s,\mathcal{E})$ of $s$-planes in $\mathcal{E}$, which is a holomorphic fiber bundle over $\mathcal{X}$ with fiber $\text{Gr}(s,r)$. Locally if $E_\alpha$ is a $G_\alpha$-equivariant vector bundle of rank $r$ over a chart $U_\alpha$, then the Grassmannian bundle $\text{Gr}(s,\mathcal{E})$ associates to $U_\alpha$ the fiber bundle $\text{Gr}(s,E_\alpha)$ of $s$-planes in $E_\alpha$, where the fiber over a point $x \in U_\alpha$ is the Grassmannian $\text{Gr}(s,(E_\alpha)_x)$ of $s$-planes in the fiber $(E_\alpha)_x$. The bundle $\text{Gr}(s, \mathcal{E}_\alpha)$ enjoys an induced action of $G_\alpha$ on it coming from the action of $G_\alpha$ on $E_\alpha$, and the induced action is such that the natural projection onto $U_\alpha$ is $G_\alpha$-equivariant. For an embedding $\lambda : U_\alpha \to U_\beta$, the bundle isomorphism $E_\alpha \to \lambda^*E_\beta$ induces a bundle isomorphism $\text{Gr}(s, E_\alpha) \to \lambda^*\text{Gr}(s, E_\beta) \simeq \text{Gr}(s, \lambda^*E_\beta)$.  

A holomorphic subbundle $\mathcal{E}'$ of $\mathcal{E}$ of rank $s$ determines a section of the fiber bundle $\text{Gr}(s,\mathcal{E})$. In addition, any section of the bundle $\text{Gr}(s,\mathcal{E})$ determines a holomorphic subbundle $\mathcal{E}'$ of $\mathcal{E}$ which corresponds to the image of the section in $\text{Gr}(s, \mathcal{E})$. 

If $p : \text{Gr}(s, \mathcal{E}) \to \mathcal{X}$ denotes the projection map, then there is a way of pulling back the vector bundle $\mathcal{E}$ along $p$ to obtain a vector bundle $p^*\mathcal{E}$ of rank $r$ over $\text{Gr}(s, \mathcal{E})$, which is described as follows. There is an atlas of charts $(U_\alpha, G_\alpha, \pi_\alpha)$ for $\mathcal{X}$ such that $(U_\alpha \times \text{Gr}(s,r), G_\alpha, \pi_\alpha')$ is an atlas of charts for $\text{Gr}(s, \mathcal{E})$, where $\pi_\alpha'$ denotes the natural map from $U_\alpha \times \text{Gr}(s,r)$ to its image in $\text{Gr}(s, \mathcal{E})$. To each embedding of charts $\lambda : U_\alpha \to U_\beta$, there corresponds an embedding of charts $\lambda' : U_\alpha \times \text{Gr}(s, r) \to U_\beta \times \text{Gr}(s,r)$ such that the diagram 
\begin{align}\label{eqn:diagram}
\xymatrix{U_\alpha \times \text{Gr}(s,r) \ar[r]^{\lambda'} \ar[d]_{\text{pr}_1^\alpha}& U_\beta \times \text{Gr}(s,r) \ar[d]^{\text{pr}_1^\beta} \\
U_\alpha \ar[r]_{\lambda} & U_\beta,
}
\end{align}
commutes, where $\text{pr}_1^\alpha$ denotes projection onto the first factor. The vector bundle $\mathcal{E}$ associates to each $U_\alpha$ a $G_\alpha$-equivariant vector bundle $E_\alpha$, and we may consider the pullback $p^*E_\alpha = (\text{pr}_1^\alpha)^* E_\alpha$ of $E_\alpha$ along the projection $U_\alpha \times \text{Gr}(s, r)$ onto the first factor. The pullback $p^*E_\alpha$ enjoys an action of $G_\alpha$ in the following manner: if $g \in G_\alpha$ and if $\xi$ is an element of $p^*E_\alpha$ in the fiber over $(x, V) \in U_\alpha \times \text{Gr}(s,r)$, then in fact $\xi$ is an element in the fiber $(E_\alpha)_x$ and $g \cdot \xi$ is an element of the fiber $(E_\alpha)_{g \cdot x} = (p^*E_\alpha)_{g\cdot(x, V)}$. In addition, to each embedding of charts $\lambda' : U_\alpha \times \text{Gr}(s,r) \to U_\beta \times \text{Gr}(s,r)$, there corresponds a bundle isomorphism $p^*E_\alpha \to (\lambda')^*p^*E_\beta$ which is described as $(\text{pr}_1^\alpha)^* \lambda_*$ where $\lambda_* : E_\alpha \to \lambda^*E_\beta$ is the bundle isomorphism induced by $\lambda$. This makes sense because there is an isomorphism of bundles $(\lambda')^*p^*E_\beta \simeq (\text{pr}_1^\alpha)^*\lambda^*E_\beta$ by the commutativity of the diagram \eqref{eqn:diagram}. 

\begin{remark}\label{rmk:pullback}
We remark that using the previous construction, it is actually possible to pull back a vector bundle $\mathcal{E}$ over $\mathcal{X}$ along the projection map $\mathcal{E}' \to \mathcal{X}$ of \emph{any} fiber bundle $\mathcal{E}'$ over $\mathcal{X}$.  However, it is not immediately clear that this construction of the pullback is readily available for each smooth map of orbifolds $\mathcal{X}' \to \mathcal{X}$. In Section 4.4 of \cite{cr}, the authors introduce the notion of a ``good'' smooth map of orbifolds, and using this notion, they show that a vector bundle may be pulled back along such maps. In particular, the projection map $\mathcal{E}' \to \mathcal{X}$ for a fiber bundle is a ``good'' map, so their construction applies in this situation, as we have just described. 
\end{remark}

There is a universal subbundle $\mathcal{S}$ of $p^*\mathcal{E}$ of rank $s$ over $\text{Gr}(s, \mathcal{E})$ described as the incidence correspondence in the usual way.

If $\mathcal{Y}$ is an analytic subvariety of $\text{Gr}(s, \mathcal{E})$ and $\mathcal{F}$ is a coherent sheaf over $\mathcal{Y}$, then there is a way of pushing forward the sheaf via the restriction of the projection $p : \text{Gr}(s, \mathcal{E}) \to \mathcal{X}$ to $\mathcal{Y}$ to obtain a coherent sheaf $p_*\mathcal{F}$ on $\mathcal{X}$ as follows. The subvariety $\mathcal{Y}$ associates to each chart $U_\alpha \times \text{Gr}(s, r)$ a $G_\alpha$-invariant subvariety $V_\alpha \subset U_\alpha \times \text{Gr}(s,r)$ and the sheaf $\mathcal{F}$ associates a $G_\alpha$-equivariant sheaf $\mathcal{F}_\alpha$ over $V_\alpha$.  We may consider the pushforward $p_*\mathcal{F}_\alpha$ onto $U_\alpha$ using the projection $\text{pr}_1^\alpha$ onto the first factor. The resulting sheaf $p_*\mathcal{F}_\alpha$ is $G_\alpha$-equivariant since $\text{pr}_1^\alpha$ is. For each embedding $\lambda : U_\alpha \to U_\beta$ of charts, there is an isomorphism of sheaves $p_*\mathcal{F}_\alpha \to \lambda^*p_*\mathcal{F}_\beta$ described as $(\text{pr}_1^\alpha)_*\tau_{\lambda'}$, where $\tau_{\lambda'}$ denotes the isomorphism of sheaves $\mathcal{F}_\alpha \to \lambda'^*\mathcal{F}_\beta$. This makes sense because there is an isomorphism of sheaves $\lambda^*p_*\mathcal{F}_\beta \simeq (\text{pr}_1^\alpha)_* \lambda'^*\mathcal{F}_\beta$ by the commutativity of the following diagram (compare to \eqref{eqn:diagram}): 
\begin{align*}
\xymatrix{
V_\alpha \ar[r]^{\lambda'|_{V_\alpha}} \ar[d] & V_\beta \ar[d] \\
U_\alpha \times \text{Gr}(s,r) \ar[r]^{\lambda'} \ar[d]_{\text{pr}_1^\alpha}& U_\beta \times \text{Gr}(s,r) \ar[d]^{\text{pr}_1^\beta} \\
U_\alpha \ar[r]^{\lambda} & U_\beta.
}
\end{align*}
In addition, the resulting sheaf is coherent by the Grauert direct image theorem \cite{grauert} because the maps $\text{pr}_1^\alpha$ are proper maps between complex spaces (as $\mathcal{Y}$ is compact and properness is preserved under base change).

By a rational map from a holomorphic orbifold $\mathcal{X}$ into another $\mathcal{X}'$ we mean we are given an analytic subvariety  $\mathcal{V}$ of codimension at least $2$ or more in $\mathcal{X}$ together with a holomorphic map from $\mathcal{X} \setminus \mathcal{V}$ into $\mathcal{X}'$.

We assert that a rational section of $\text{Gr}(s,\mathcal{E})$ over $\mathcal{X}$ determines a coherent subsheaf of $\mathcal{E}$ in the following manner. Let $\mathcal{Y}$ denote the closure of the image of the section in $\text{Gr}(s, \mathcal{E})$. The restriction of the universal bundle $\mathcal{S}$ to $\mathcal{Y}$ is a coherent sheaf of rank $s$ over $\mathcal{Y}$. In addition, as a closed subset of a compact space, $\mathcal{Y}$ is compact itself, and so the projection of $\mathcal{Y}$ onto $\mathcal{X}$ is proper. Pushing forward the restriction $\mathcal{S}|_{\mathcal{Y}}$ of the universal bundle via the projection of $\mathcal{Y}$ onto $\mathcal{X}$, we obtain a sheaf $\mathcal{F}$ over $\mathcal{X}$, which is coherent by our above observations. 

\medskip

\noindent \emph{Proof of Lemma \ref{lem:weak}}.
A weakly holomorphic subbundle $\Pi$ determines a map from a set of full measure in $\mathcal{X}$ to the total space of the bundle $\text{Gr}(s, \mathcal{E})$. Uhlenbeck-Yau demonstrate how the assumptions $\Pi = \Pi^2 = \Pi^* \in L_1^2$ and $(I_{\mathcal{E}} - \Pi) \bar{\partial}\Pi = 0$ imply that $\Pi$ extends to a rational section of the bundle $\text{Gr}(s, \mathcal{E})$. The previous discussion explains how this rational section furnishes a coherent subsheaf of $\mathcal{E}$ over $\mathcal{X}$. \hfill $\Box$

\section{Simpson's method}

This section is devoted to proving Lemma \ref{lem:proper} following an approach from \cite{simpson}. We first require a somewhat technical estimate relating the $C^0$-norm to the $L^2$-norm for solutions to the heat flow. This result is intended to replace Assumption 3 from \cite{simpson}. It is important to note that in this section all inner products and norms are to be computed with respect to $K$ unless otherwise indicated. In particular, this means that we will use the notation $L^p$ to denote the $L^p$-norm of a section with respect to the fixed metric $K$. 

\begin{lemma}\label{lem:C^0control}
Fix a metric $K$. Then there are positive constants $C_1, C_2$ such that the following is true. Let $H_t$ be a solution to the heat flow with initial condition $H_0 = K$ and write $H_t = Ke^{s_t}$ for a path $t \mapsto s_t$ of self-adjoint endomorphisms with initial condition $s_t = 0$. Then for any $t$ we have 
\[
\sup_{\mathcal{X}}|s_t|_K \leqslant C_1 + C_2 \norm{|s_t|^2_K}_{L^2}^{1/2}.
\]
\end{lemma}

\begin{proof}
In the course of the proof, we let $C_1, C_2, \ldots$ denote constants that are independent of $t$ but which may vary from step to step. Recall from Lemma \ref{lem:curvaturerelations} (v),  we know that 
\[
\Delta |s_t|_K^2 = \langle 2i \Lambda (F_t - F_K), s_t \rangle - \langle i \Lambda F_K s_t, s_t \rangle - |d_Ks_t|^2  
\]
and hence because the last term is a square we find
\[
\Delta |s_t|_K^2 \leqslant \langle 2i \Lambda (F_t - F_K), s_t \rangle - \langle i \Lambda F_K s_t, s_t \rangle.
\]
Proposition \ref{prop:decreasing} (with the Schwarz inequality) implies there are positive constants $C_1, C_2$ such that 
\begin{align*}
\Delta |s_t|_K^2 &\leqslant C_1|s_t| + C_2|s_t|^2 \\
&\leqslant C_1 + C_2 |s_t|^2.
\end{align*}

Let $p_t$ be a point where $|s_t|_K^2$ achieves its maximum. Let $G_{p_t} \in L_2^2$ be Green's function for $\Delta$. Green's formula gives that 
\[
|s_t|^2(p_t) = \frac{1}{\text{vol}(\mathcal{X})}\int_{\mathcal{X}}|s_t|^2 \frac{\omega^n}{n!} + \int_{\mathcal{X}} G_{p_t} \Delta|s_t|^2 \frac{\omega^n}{n!}.
\] 
Because $G_{p_t}$ is bounded from below, we may assume by shifting by a constant that $G_{p_t}$ is positive (c.f. \cite{aubinsome}). Moreover, we may assume that $G_{p_t}$ is square integrable. In fact, because $\mathcal{X}$ is compact, there is a constant $C$ (independent of $t$) such that $\norm{G_{p_t}}_{L^2} \leqslant C$. Using the previous paragraph and the Schwarz inequality we find that we have an inequality of the form 
\[
|s_t|^2(p_t) \leqslant C_1\norm{|s_t|^2}_{L^1} + C_2 + C_3 \norm{|s_t|^2}_{L^2}.
\]
The inclusion of $L^2$ into $L^1$ implies that 
\[
|s_t|^2(p_t) \leqslant  C_1 + C_2 \norm{|s_t|^2}_{L^2}.
\]

Now, we also know that  
\[
(\sup_{\mathcal{X}}|s_t|)^2 \leqslant 1 + \sup_{\mathcal{X}}|s_t|^2,
\]
and so in conjunction with the previous paragraph we have 
\[
\sup |s_t| \leqslant C_1 + C_2 \norm{|s_t|^2}_{L^2}^{1/2},
\]
as desired. 
\end{proof}

It is also prudent to understand the variation of $M_K$ along a path of the form $t \mapsto Ke^{ts}$ for a fixed endomorphism $s$ of $\mathcal{E}$ that is self-adjoint with respect to $K$ and which satisfies $\int_{\mathcal{X}} \text{Tr}(s) \omega^n = 0$. We will see that the functional $M_K$ is convex along such a path.

\begin{proposition}\label{prop:M_Kconvex}
If $H_t = Ke^{ts}$ for an endomorphism $s$ with $\int_{\mathcal{X}} \textnormal{Tr}(s) \omega^n = 0$  which is self-adjoint with respect to $K$, then 
\begin{align*}
\frac{\partial}{\partial t}M_K(Ke^{ts}) = 2i \int_{\mathcal{X}} \textnormal{Tr}(s F_t) \wedge \frac{\omega^{n-1}}{(n-1)!}
\end{align*}
and 
\begin{align*}
\frac{\partial^2}{\partial t^2} M_K(Ke^{ts}) =  2 \int_{\mathcal{X}} |\bar{\partial}s|_{H_t}^2 \frac{\omega^n}{n!}.
\end{align*}
\end{proposition}

\begin{proof}
Note that along this path, we have $\dot{H}_t = H_ts$, and so using Proposition \ref{prop:M_K} (i), we readily verify the formula of the first variation (using that $\int_{\mathcal{X}}\text{Tr}(s) \omega^n = 0$). Upon taking another derivative, we find 
\[
\frac{\partial^2}{\partial t^2} M_K(Ke^{ts}) = 2i \int_{\mathcal{X}} \textnormal{Tr}(s \dot{F}_t) \wedge \frac{\omega^{n-1}}{(n-1)!}.
\]
But because $H_t = Ke^{ts}$, we have that the curvatures are related by 
\begin{align*}
F_t &= F_K +  \bar{\partial} (H_t^{-1} \partial_{K} H_t) = F_K + t\bar{\partial}(\partial_{K} (s))
\end{align*}
so that $\dot{F}_t = \bar{\partial} \partial_{K} s$, from which we find 
\[
\frac{\partial^2}{\partial t^2} M_K(Ke^{ts}) = 2i \int_{\mathcal{X}} \textnormal{Tr}(s \wedge \bar{\partial} \partial_{K}s) \wedge \frac{\omega^{n-1}}{(n-1)!}.
\]
One integration by parts shows that 
\[
\frac{\partial^2}{\partial t^2} M_K(Ke^{ts}) = -2i \int_{\mathcal{X}} \textnormal{Tr}(\bar{\partial} s \wedge \partial_{K}s) \wedge \frac{\omega^{n-1}}{(n-1)!}.
\]
And another shows that 
\[
\frac{\partial^2}{\partial t^2} M_K(Ke^{ts}) = -2i \int_{\mathcal{X}} \textnormal{Tr}(\partial_K\bar{\partial} s \wedge s) \wedge \frac{\omega^{n-1}}{(n-1)!}.
\]
(Note that the sign is preserved here because $\bar{\partial}s$ is a $1$-form.) Then the K\"ahler identity $\bar{\partial}^* = - i\Lambda \partial$ shows that 
\[
\frac{\partial^2}{\partial t^2} M_K(Ke^{ts}) = 2 \int_{\mathcal{X}} \textnormal{Tr}(\bar{\partial}^*\bar{\partial} s \wedge s) \wedge \frac{\omega^n}{n!}.
\]
Now using that $s$ is self-adjoint with respect to $H_t$, we find that this is equal to 
\[
\frac{\partial^2}{\partial t^2} M_K(Ke^{ts}) = 2 \int_{\mathcal{X}} \langle \bar{\partial}^* \bar{\partial}s, s \rangle_{H_t} \frac{\omega^n}{n!}.
\]
This is equivalent to the desired formula. 
\end{proof}

This proposition allows one to obtain a slightly different expression for the functional $M_K$, which can be found for example in \cite{simpson, siu, jacob, donaldson2}. Indeed, let us write $M(t)$ for the value $M(t) = M_K(Ke^{ts})$. Then with this convention, we have from the previous proposition that $M'(0)$ is given by 
\[
M'(0) = 2i\int_{\mathcal{X}}\text{Tr}(s F_K) \wedge \frac{\omega^{n-1}}{(n-1)!}.
\]
The fundamental theorem of Calculus in conjunction with the previous proposition gives   
\[
M'(t) = 2i\int_{\mathcal{X}}\text{Tr}(s F_K) \wedge \frac{\omega^{n-1}}{(n-1)!} + 2 \int_{0}^t  \int_{\mathcal{X}} |\bar{\partial}s|_{H_u}^2 \frac{\omega^n}{n!} du.
\]
The condition $M(0) = 0$ implies then that $M(1)$ is given by an additional integration 
\[
M_K(Ke^s) = M(1) = 2i\int_{\mathcal{X}}\text{Tr}(s F_K) \wedge \frac{\omega^{n-1}}{(n-1)!} + 2 \int_0^1 \left(\int_{0}^t  \int_{\mathcal{X}} |\bar{\partial}s|_{H_u}^2 \frac{\omega^n}{n!} du \right) dt.
\]
We now follow \cite{donaldson2} to write the second term on the right-hand side with a local expression involving frames. 

Let us fix a smooth unitary (with respect to $K$) frame for $\mathcal{E}$ for which the matrix of $s$ with respect to this frame is diagonal with eigenvalues $\lambda_1, \ldots, \lambda_r$. (The matrix of $\bar{\partial} s$ may not be diagonal because the frame is only \emph{smooth}.) With these conventions, then the integrand of the second equality in the previous proposition becomes 
\begin{align*}
|\bar{\partial}s|_{H_t}^2 &= (\bar{\partial}s)_{\alpha}^\beta \overline{(\bar{\partial}s)_{\gamma}^\rho} (H_t)^{\alpha \gamma}(H_t)_{\beta \rho} \\
&= (\bar{\partial}s)_{\alpha}^\beta \overline{(\bar{\partial}s)_{\gamma}^\rho} (e^{-t\lambda_\alpha} \delta^{\alpha\gamma})( e^{t\lambda_\beta}\delta_{\beta \rho}) \\
&= \sum_{\alpha,\beta} |(\bar{\partial}s)_{\alpha}^\beta|^2 e^{(\lambda_\beta - \lambda_\alpha)t}.
\end{align*}
Integrating once we obtain 
\[
\int_0^t |\bar{\partial}s|_{H_u}^2 du = \sum_{\alpha,\beta} |(\bar{\partial}s)_{\alpha}^\beta|^2 \frac{e^{(\lambda_\beta - \lambda_\alpha)t}-1}{\lambda_\beta - \lambda_\alpha}.
\]
And integrating again gives 
\[
\int_0^1 \left(\int_0^t  |\bar{\partial}s|_{H_u}^2 du \right) dt = \sum_{\alpha,\beta} |(\bar{\partial}s)_\alpha^\beta|^2 \frac{e^{\lambda_\beta - \lambda_\alpha} - (\lambda_\beta - \lambda_\alpha) - 1}{(\lambda_\beta - \lambda_\alpha)^2}.
\]
What we have shown therefore is that 
\begin{align}\label{eqn:M_K}
M(Ke^s,K) = 2i \int_{\mathcal{X}} \text{Tr}(s\Lambda F_K) \frac{\omega^n}{n!} + 2 \int_{\mathcal{X}} \sum_{\alpha, \beta} |(\bar{\partial} s)_\alpha^\beta|^2 \frac{e^{\lambda_\beta - \lambda_\alpha} - (\lambda_\beta - \lambda_\alpha) -1}{(\lambda_\beta - \lambda_\alpha)^2} \frac{\omega^n}{n!},
\end{align}
where the summand is interpreted as $\frac{1}{2}|(\bar{\partial} s)_\alpha^\beta|^2$ if $\alpha = \beta$. 

Following \cite{siu} and \cite[Lemma 24]{donaldson2} it is then possible to obtain the following estimate, which we won't really need, but which we collect for completeness. 

\begin{corollary}\label{cor:M_Kest}
For any endomorphism $s$ with $\int_{\mathcal{X}} \textnormal{Tr}(s) \omega^n = 0$ that is self-adjoint with respect to $K$, we have 
\[
\norm{D_K s}_{L^1}^2 \leqslant 2 (\sqrt{2}\norm{s}_{L^1} + \textnormal{vol}(\mathcal{X}))\left(M_K(Ke^s) - 2i \int_{\mathcal{X}} \textnormal{Tr}(s\Lambda F_K) \frac{\omega^n}{n!}  \right). 
\]
\end{corollary}

\begin{proof}
For any real number $u$, we have the following inequality 
\[
\ \frac{1}{2 \sqrt{u^2 + 1}} \leqslant \frac{e^u - u - 1}{u^2} ,
\]
which is verified, for example, in \cite{siu}. From this it follows immediately upon setting $u = \lambda_\beta - \lambda_\alpha$ that  
\[
\frac{1}{2 \sqrt{(\lambda_\beta - \lambda_\alpha)^2 + 1}} \leqslant \frac{e^{\lambda_\beta - \lambda_\alpha} - (\lambda_\beta -\lambda_\alpha)  - 1}{(\lambda_\beta - \lambda_\alpha)^2} .
\]
The inequality  
\[
2(\lambda_\beta^2 + \lambda_\alpha^2) = (\lambda_\beta - \lambda_\alpha)^2 + (\lambda_\beta + \lambda_\alpha)^2 \geqslant (\lambda_\beta - \lambda_\alpha)^2
\]
implies also that 
\[
 \frac{1}{2 \sqrt{2(\lambda_\beta^2 + \lambda_\alpha^2) + 1}} \leqslant \frac{e^{\lambda_\beta - \lambda_\alpha} - (\lambda_\beta -\lambda_\alpha)  - 1}{(\lambda_\beta - \lambda_\alpha)^2} .
\]
And finally the inequality   
\[
|s|_K^2 = \sum_{\alpha}\lambda_\alpha^2 \geqslant \lambda_\beta^2 + \lambda_\alpha^2
\]
implies 
\[
 \frac{1}{2 \sqrt{2|s|_K^2 + 1}} \leqslant \frac{e^{\lambda_\beta - \lambda_\alpha} - (\lambda_\beta -\lambda_\alpha)  - 1}{(\lambda_\beta - \lambda_\alpha)^2}.
\]
Now using that $|D_Ks|_K^2 = 2|\bar{\partial}s|_K^2$ (because $s$ is self-adjoint), we find that  
\[
\frac{|D_K s|_K^2}{4 \sqrt{2|s|_K + 1}} = \frac{|\bar{\partial} s|_K^2}{2 \sqrt{2|s|_K + 1}}  \leqslant  \sum_{\alpha, \beta} |\bar{\partial} s_\alpha^\beta|^2 \frac{e^{\lambda_\beta - \lambda_\alpha} - (\lambda_\beta - \lambda_\alpha) -1}{(\lambda_\beta - \lambda_\alpha)^2}.  
\] 
Upon integrating over $\mathcal{X}$ and using formula \eqref{eqn:M_K}, we find  
\begin{align}\label{ineq:M_K}
\frac{1}{2} \int_{\mathcal{X}} \frac{|D_K s|_K^2}{\sqrt{2|s|_K + 1}} \frac{\omega^n}{n!} \leqslant M_K(Ke^s) - 2i \int_{\mathcal{X}} \text{Tr}(s \Lambda F_K) \frac{\omega^n}{n!}.
\end{align}
On the other hand, if we write 
\[
D_Ks = \frac{D_Ks}{(2|s|_K + 1)^{1/4}} (2|s|_K + 1)^{1/4}
\]
and use the Cauchy inequality, we find that 
\begin{align*}
\left(\int_{\mathcal{X}} |D_K s|_K \frac{\omega^n}{n!} \right)^2 &\leqslant \left(\int_{\mathcal{X}}\frac{|D_K s|_K^2}{\sqrt{2|s|_K + 1}} \frac{\omega^n}{n!} \right)\left(\int_{\mathcal{X}} (2|s|_K^2 + 1)^{1/2}\frac{\omega^n}{n!} \right) \\
&\leqslant \left(\int_{\mathcal{X}}\frac{|D_K s|_K^2}{\sqrt{2|s|_K + 1}} \frac{\omega^n}{n!} \right)\left(\int_{\mathcal{X}} (\sqrt{2}|s|_K + 1)\frac{\omega^n}{n!} \right) \\
&\leqslant  \left(\int_{\mathcal{X}}\frac{|D_K s|_K^2}{\sqrt{2|s|_K + 1}} \frac{\omega^n}{n!} \right)\left( \sqrt{2} \norm{s}_{L^1} + \text{vol}(\mathcal{X}) \right),
\end{align*}
and then using \eqref{ineq:M_K} we conclude 
\[
\norm{D_K s}_{L^1}^2 \leqslant 2 \left(\sqrt{2}\norm{s}_{L^1} + \text{vol}(\mathcal{X})\right)\left(M_K(Ke^s) - 2i \int_{\mathcal{X}} \textnormal{Tr}(s\Lambda F_K) \frac{\omega^n}{n!}  \right),
\]
as desired. 
\end{proof}

\begin{notation}\label{notation1}
Let us follow \cite{simpson} to introduce briefly some notation that will allow us to express formula \eqref{eqn:M_K} in a global manner. 

For a smooth function $\varphi : \mathbb{R} \to \mathbb{R}$, an endomorphism $s$ of $\mathcal{E}$ that is self-adjoint with respect to $K$, we let $\varphi(s)$ denote the endomorphism described in the following manner. If $\{e_1, \ldots, e_r\}$ is a smooth unitary (with respect to $K$) frame for $E$ with respect to which $s$ is diagonal with entries $\lambda_1, \ldots, \lambda_r$, then $\varphi(s)$ is the endomorphism with diagonal entries $\varphi(\lambda_1), \ldots, \varphi(\lambda_r)$. 

In addition, for a smooth function $\Phi : \mathbb{R} \times \mathbb{R} \to \mathbb{R}$ of two variables, a self-adjoint endomorphism $s \in \text{End}(\mathcal{E}, K)$, and an endomorphism $A \in \text{End}(\mathcal{E})$, we let $\Phi(s)(A)$ denote the endomorphism of $\mathcal{E}$ described in the following manner. If $\{e_1, \ldots, e_r\}$ is a smooth unitary (with respect to $K$) frame of $E$ with respect to which $s$ is diagonal with eigenvalues $\lambda_1, \ldots, \lambda_r$ and $A$ has the local expression $A = A_{\alpha}^\beta e^\alpha \otimes e_\beta$ where $e^\alpha$ is the frame dual to $e_\beta$, then the endomorphism $\Phi(s)(A)$ has local expression 
\[
\Phi(s)(A) = \sum_{\alpha,\beta}\Phi(\lambda_\alpha, \lambda_\beta) A_\alpha^\beta e^\alpha \otimes e_\beta.
\]
\end{notation}

The construction $\Phi$ enables one to express the derivatives of construction $\varphi$ in the following way. 

\begin{lemma}\label{lem:chainrule}
Given a $\varphi : \mathbb{R} \to \mathbb{R}$, if we set $d\varphi : \mathbb{R} \times \mathbb{R} \to \mathbb{R}$ to be the difference quotient defined by  
\[
d\varphi(u,v) = \frac{\varphi(u) - \varphi(v)}{u-v}
\]
for $u \ne v$ and $d\varphi(u,u) = \frac{d}{du}\varphi(u)$ along the diagonal, then we have 
\[
\bar{\partial}(\varphi(s)) = d\varphi(s)(\bar{\partial}s).
\]
In addition, if $\Phi : \mathbb{R} \times \mathbb{R} \to \mathbb{R}$ is any smooth function which agrees with $d\varphi$ along the diagonal, then 
\[
\textnormal{Tr}(\Phi(s)(\bar{\partial}s)) = \textnormal{Tr}(d\varphi(s)(\bar{\partial}s)).
\]
\end{lemma}

\begin{proof}
Let $e_\alpha$ be a smooth unitary frame for $\mathcal{E}$ with respect to which a local expression for $s$ is 
\[
s = \sum_{\alpha} \lambda_\alpha e^\alpha \otimes e_\alpha
\] 
for some local smooth functions $\lambda_\alpha$. 
Let us also write 
\[
\bar{\partial} e_\alpha = \theta_{\alpha}^\beta e_\beta
\] 
for some local $(0,1)$-forms $\theta_\alpha^\beta$. The relation $e^\alpha(e_\beta) = \delta_{\beta}^\alpha$ implies then that we also have  
\[
\bar{\partial} e^\alpha = -\theta_\beta^\alpha e^\beta.
\]
It follows that a local expression for $\bar{\partial}s$ is given by 
\begin{align*}
\bar{\partial}s &= (\bar{\partial} \lambda_\alpha) e^\alpha \otimes e_\alpha +  \sum_{\alpha, \beta} (-\lambda_\alpha \theta^\alpha_\beta e^\beta \otimes e_\alpha + \lambda_\alpha \theta_\alpha^\beta e^\alpha \otimes e_\beta) \\
&= (\bar{\partial} \lambda_\alpha) e^\alpha \otimes e_\alpha + \sum_{\alpha \ne \beta} (\lambda_\alpha -\lambda_\beta) \theta_\alpha^\beta e^\alpha \otimes e_\beta,
\end{align*}
which means precisely that the coefficients of $\bar{\partial}s$ are given by 
\[
(\bar{\partial}s)_\alpha^\beta = \begin{cases}
(\lambda_\alpha -\lambda_\beta) \theta_\alpha^\beta & \alpha \ne \beta  \\
\bar{\partial} \lambda_\alpha & \alpha = \beta
\end{cases}.
\] 
More generally, a similar computation shows that the coefficients of $\bar{\partial}(\varphi(s))$ are given by
\[
(\bar{\partial}(\varphi(s))_\alpha^\beta = \begin{cases}
(\varphi \circ \lambda_\alpha - \varphi\circ \lambda_\beta) \theta_\alpha^\beta & \alpha \ne \beta  \\
\bar{\partial} (\varphi \circ \lambda_\alpha ) & \alpha = \beta
\end{cases}.
\]  On the other hand, by definition, the endomorphism $d\varphi(s)(\bar{\partial}s)$ has coefficients
\begin{align*}
(d\varphi(s)(\bar{\partial}s))_\alpha^\beta &= d\varphi(\lambda_\alpha, \lambda_\beta) (\bar{\partial}s)_\alpha^\beta \\
&= \begin{cases}
\frac{\varphi\circ\lambda_\alpha - \varphi\circ \lambda_\beta}{\lambda_\alpha - \lambda_\beta} (\lambda_\alpha - \lambda_\beta)\theta_\alpha^\beta & \alpha \ne \beta  \\
\varphi'(\lambda_\alpha) \bar{\partial}\lambda_\alpha & \alpha = \beta
\end{cases} .
\end{align*}
Comparing coefficients, we find that the first part of the lemma follows.

For the second part about the trace, suppose that $\Phi :\mathbb{R} \times \mathbb{R} \to \mathbb{R}$ is any smooth function which agrees with $d\varphi$ along the diagonal. Then the trace of $\Phi(s)(\bar{\partial}s)$ is given by 
\begin{align*}
\text{Tr}(\Phi(s)(\bar{\partial}s)) &= \sum_{\alpha} \Phi(\lambda_\alpha, \lambda_\alpha) (\bar{\partial}s)_\alpha^\alpha \\
&= \sum_{\alpha} \varphi'(\lambda_\alpha) \bar{\partial}\lambda_\alpha \\
&= \text{Tr}(d\varphi(s)(\bar{\partial}s)),
\end{align*} 
as desired. 
\end{proof}

With these conventions, we see that formula \eqref{eqn:M_K} is equivalent to 
\begin{align}\label{eqn:M_Kinv}
M_K(Ke^s) = 2i \int_{\mathcal{X}} \text{Tr}(s\Lambda F_K) \frac{\omega^n}{n!} + 2 \int_{\mathcal{X}} \langle \Psi(s)(\bar{\partial} s), \bar{\partial} s \rangle_{K} \frac{\omega^n}{n!},
\end{align}
where $\Psi : \mathbb{R} \times \mathbb{R} \to \mathbb{R}$ is the function 
\[
\Psi(u,v) = \frac{e^{v - u} - (v - u) - 1}{(v - u)^2},
\]
which is extended continuously (and smoothly) along the diagonal by requiring that $\Psi(u,u) = 1/2.$

The construction $\Phi$ extends to $L^p$-spaces of endomorphisms in the following way. Because $\Phi$ is smooth, there is a positive constant $C$ depending on $\Phi$ such that we have the pointwise estimate
\[
|\Phi(s)(A)|_K \leqslant C |s|_K |A|_K
\]
for any endomorphism $A$ and self-adjoint endomorphism $s$. Given any $1 \leqslant p < q$, if $r$ is the number $1/p = 1/q + 1/r$, then H\"older's inequality implies that 
\[
\norm{|s|_K|A|_K}_{L^p} \leqslant \norm{s}_{L^r} \norm{A}_{L^q}.
\]
It follows that  for $1 \leqslant p < q$, given a self-adjoint endomorphism $s \in L^r(\text{End}(\mathcal{E}))$, the construction $A \mapsto \Phi(s)(A)$ describes a bounded linear operator 
\[
\Phi(s) : L^q(\text{End}(\mathcal{E})) \to L^p(\text{End}(\mathcal{E}))
\]
whose norm satisfies 
\[
\norm{\Phi(s)} \leqslant C \norm{s}_{L^r}.
\]
In this way,  we may think of $\Phi$ as a mapping 
\[
\Phi : L^r(\text{End}(\mathcal{E}, K)) \to \text{Hom}(L^q(\text{End}(\mathcal{E})) , L^p(\text{End}(\mathcal{E}))).
\]
Moreover, it also follows that if $s_k$ is a sequence that converges in the $L^r$-norm to $s_\infty$, then the sequence $\Phi(s_k)$ of operators converges in the operator norm to $\Phi(s_\infty)$. We summarize in the following proposition.

\begin{proposition}\label{prop:operatornorm}
For $1 \leqslant p < q$, the construction $\Phi$ describes a continuous mapping 
\[
\Phi : L^r(\textnormal{End}(\mathcal{E}, K)) \to \textnormal{Hom}(L^q(\textnormal{End}(\mathcal{E})) , L^p(\textnormal{End}(\mathcal{E}))),
\]
where $r$ is the number satisfying $1/p = 1/q + 1/r$. 
\end{proposition}

\noindent \emph{Proof of Lemma \ref{lem:proper}}. Assuming the properness condition in Definition \ref{def:proper} is violated, we will construct explicitly a weakly holomorphic subbundle that is destabilizing. 

For a solution $H_t$ to the heat flow with initial condition $H_0 = K$, let us write $H_t = Ke^{s_t}$ for a path of self-adjoint endomorphisms $s_t$ with initial condition $s_0 = 0.$ 

We first claim that we have $\int_{\mathcal{X}} \text{Tr}(s_t) \omega^n = 0$ along the path $t \mapsto s_t$. Indeed, the heat equation \eqref{eqn:heat} implies that 
\[
\dot{s}_t = -\frac{i}{2} (\Lambda F_t - \lambda I_\mathcal{E}).
\] 
Upon taking the trace and integrating over $\mathcal{X}$, we find that the right-hand side vanishes, and so the quantity $\int_{\mathcal{X}} \text{Tr}(s_t) \omega^n$ must be constant. The initial condition $s_0 = 0$ implies that this constant must be zero, as desired.

Now assume the properness condition of Definition \ref{def:proper} is violated. By Lemma \ref{lem:C^0control}, we can find $s_t$ contradicting the estimate with $\norm{|s_t|_K^2}_{L^2}$ arbitrarily large, or else the resulting bound on the $C^0$-norm would make the estimate of Definition \ref{def:proper} hold trivially after adjusting $C_1$. We thus have a sequence of times $t_k$ and corresponding self-adjoint endomorphisms $s_k$ whose $L^2$-norms $\norm{|s_k|^2}_{L^2}$ tend to $\infty$ and which satisfy 
\begin{align}\label{uncoerc}
\norm{|s_k|^2}_{L^2}^{1/2} \geqslant k M_K(Ke^{s_k}). 
\end{align}

Let us define a sequence of normalized endomorphisms $u_k = \ell_k^{-1} s_k$, where $\ell_k$ is the number 
\[
\ell_k = \norm{|s_k|^2}_{L^2}^{1/2}.
\] 
Note that the $u_k$ are indeed normalized in the sense that $\norm{|u_k|^2}_{L^2}^{1/2} = 1$. The uniform estimate of Lemma \ref{lem:C^0control} implies that 
\[
\ell_k \sup_{\mathcal{X}}|u_k| \leqslant C_1 + C_2 \ell_k \norm{|u_k|^2}_{L^2}^{1/2},
\]
and so we obtain a uniform $C^0$-bound on the sequence $u_k$. 

We now prove the following useful lemma. 

\begin{lemma}\label{lem:essential}
After passing to a subsequence, we may assume that the sequence $u_k$ converges to $u_\infty$ weakly in $L_1^2$. If $\Phi : \mathbb{R} \times \mathbb{R} \to \mathbb{R}$ is a positive smooth function satisfying $\Phi(u,v) < (u - v)^{-1}$ whenever $u > v$, then 
\[
i \int_{\mathcal{X}} \textnormal{Tr}(u_\infty \Lambda F_K) \frac{\omega^n}{n!} + \int_{\mathcal{X}} \langle \Phi(u_\infty)(\bar{\partial} u_\infty), \bar{\partial} u_\infty\rangle_K \frac{\omega^n}{n!} \leqslant 0,
\]
where $\Phi(u_\infty)(\bar{\partial} u_\infty)$ is the endomormorphism of $\mathcal{E}$ constructed as in Notation \ref{notation1}. 
\end{lemma}

\noindent \emph{Proof of Lemma \ref{lem:essential}}. Condition \eqref{uncoerc} can be written as 
\[
2i \ell_k  \int_{\mathcal{X}} \text{Tr}(u_k\Lambda F_K) \frac{\omega^n}{n!} + 2 \ell_k^2  \int_{\mathcal{X}} \langle \Psi(\ell_k u_k)(\bar{\partial} u_k), \bar{\partial}u_k \rangle_K  \frac{\omega^n}{n!} \leqslant \frac{1}{k}\ell_k.
\]
As $\ell \to \infty$, the expression 
\[
\ell\Psi(\ell u, \ell v) =  \frac{\ell e^{\ell (v - u)} - \ell^2(v - u) - \ell}{\ell^2 (v - u)^2} 
\] 
increases monotonically to $(u - v)^{-1}$ for $u > v$ and to $\infty$ for $u \leqslant v$. 

Fix $\Phi$ as in the statement of the lemma. Because the construction $\Phi(u_k)$ depends only on the eigenvalues of $u_k$ and these are bounded uniformly in $k$ (by the $C^0$-bound on the sequence), we may assume that $\Phi$ is compactly supported. Then the assumption on $\Phi$ guarantees that $\Phi(u,v) < \ell \Psi(\ell u, \ell v)$ for $\ell$ sufficiently large. It follows from the previous paragraph that for $k$ sufficiently large, we have
\begin{align}\label{eqn:essential}
i  \int_{\mathcal{X}} \text{Tr}(u_k\Lambda F_K) \frac{\omega^n}{n!} +  \int_{\mathcal{X}} \langle \Phi(u_k)(\bar{\partial}u_k), \bar{\partial} u_k\rangle_K  \frac{\omega^n}{n!} \leqslant \frac{1}{2k}.
\end{align}
The $C^0$-bound on the sequence $u_k$ implies that the operator norms of $\Phi(u_k)$ are bounded uniformly, and hence we obtain from this inequality a uniform bound on $\norm{\bar{\partial}u_k}_{L^2}$. Therefore we may choose a subsequence so that $u_k \to u_\infty$ weakly in $L_1^2$ (and strongly in $L^2$). 

Moreover, the fact that we have a uniform $C^0$-bound on the sequence $u_k$ implies that the sequence $u_k$ converges to $u_\infty$ in $L^r$ for any $r > 2$. Indeed, let us write $b$ for a uniform $C^0$-bound for the sequence $u_k$. Then we compute that 
\begin{align*}
\norm{u_k - u_j}_{L^r}^r &= \int_{\mathcal{X}}|u_k - u_j|_K^r \frac{\omega^n}{n!} \\
&= \int_{\mathcal{X}} |u_k - u_j|^{r-2}_K|u_k - u_j|_K^{2}  \frac{\omega^n}{n!} \\
&\leqslant (2b)^{r-2} \int_{\mathcal{X}}|u_k - u_j|_K^{2} \frac{\omega^n}{n!} \\
&= (2b)^{r-2} \norm{u_k - u_j}_{L^2}^2.
\end{align*}
This estimate implies that if the sequence $u_k$ is Cauchy in $L^2$ then it is also Cauchy in $L^r$ for $r > 2$.

The proof of this lemma would be complete if we knew we could take a limit of the inequality \eqref{eqn:essential} as $k \to \infty$. We can do so for the following reasons. Let $\epsilon > 0$ be arbitrary. Notice that
\begin{align*}
\norm{\Phi^{1/2}(u_k)(\bar{\partial} u_k)}_{L^2}^2 = \int_{\mathcal{X}} \langle \Phi(u_k)(\bar{\partial}u_k), \bar{\partial} u_k\rangle_K  \frac{\omega^n}{n!}.
\end{align*}
The mapping
\[
u \mapsto i  \int_{\mathcal{X}} \text{Tr}(u\Lambda F_K) \frac{\omega^n}{n!}
\] 
is continuous for $u \in L^2$, so the inequality \eqref{eqn:essential} implies that for $k$ sufficiently large we have
\[
i \int_{\mathcal{X}} \text{Tr}(u_\infty \Lambda F_K) \frac{\omega^n}{n!} + \norm{\Phi^{1/2}(u_k)(\bar{\partial} u_k)}_{L^2}^2 \leqslant \epsilon.
\] 
This estimate implies in particular that the sequence of numbers $\norm{\Phi^{1/2}(u_k)(\bar{\partial} u_k)}_{L^2}^2$ is bounded uniformly. 
Let $p$ be a number satisfying $1 < p < 2$, and let $r$ be the positive number such that $1/p = 1/2 + 1/r$. The inequality $1 < p$ implies that $r > 2$.  By the previous paragraph, because $r > 2$, the sequence $u_k$ converges in $L^r$. It follows from Proposition \ref{prop:operatornorm} that the sequence of operators $\Phi^{1/2}(u_k)$ converges to $\Phi^{1/2}(u_\infty)$ in the space $\textnormal{Hom}(L^2(\textnormal{End}(\mathcal{E})) , L^p(\textnormal{End}(\mathcal{E}))$.   The sequence $\bar{\partial}u_k$ is bounded in $L^2$, so we may appeal to Proposition \ref{prop:operatornorm} to find that $\Phi^{1/2}(u_k)(\bar{\partial}u_j) \to \Phi^{1/2}(u_\infty)(\bar{\partial} u_j)$ for fixed $j$ as $k \to \infty$ in $L^p$. This means that for $k$ sufficiently large we have 
\[
\norm{\Phi^{1/2}(u_\infty)(\bar{\partial} u_j)}_{L^p}^2 \leqslant \norm{\Phi^{1/2}(u_k)(\bar{\partial} u_j)}_{L^p}^2 + \epsilon,
\]
where this estimate is independent of $j$ because the sequence $\bar{\partial}u_j$ is bounded uniformly in $L^2$. 
In addition, the sequence $\Phi^{1/2}(u_\infty)(\bar{\partial} u_j)$ converges to $\Phi^{1/2}(u_\infty)(\bar{\partial} u_\infty)$ weakly in $L^p$, so by the lower semicontinuity of the norm, we find that for $j,k$ sufficiently large, we have 
\begin{align*}
\norm{\Phi^{1/2}(u_\infty)(\bar{\partial} u_\infty)}_{L^p}^2 &\leqslant \norm{\Phi^{1/2}(u_\infty)(\bar{\partial} u_j)}_{L^p}^p + \epsilon \\
 &\leqslant \norm{\Phi^{1/p}(u_k)(\bar{\partial} u_j)}_{L^p}^2 + 2\epsilon.
\end{align*}
Moreover, we have an estimate of the form 
\[
\norm{f}_{L^p} \leqslant (\text{vol}(\mathcal{X}))^{1/r} \norm{f}_{L^2}
\]
for $f \in L^2$. 
As $p \to 2$, we have $r \to \infty$, so by choosing $p$ sufficiently close to $2$, we may ensure that $(\text{vol}(\mathcal{X}))^{1/r}$ is sufficiently close to $1$. Because the sequence of numbers $\norm{\Phi^{1/2}(u_k)(\bar{\partial} u_k)}_{L^2}^2$ is bounded uniformly, we may now ensure that by taking $p$ close enough to $2$ that we have a uniform estimate of the form 
\[
\norm{\Phi^{1/2}(u_k)(\bar{\partial} u_k)}_{L^p}^{2} \leqslant \norm{\Phi^{1/2}(u_k)(\bar{\partial} u_k)}_{L^2}^2 + \epsilon.
\]
Collecting all of the above, what we have shown therefore is that for $k$ sufficiently large and for $p$ sufficiently close to $2$, we have 
\begin{align*}
&i \int_{\mathcal{X}} \text{Tr}(u_\infty \Lambda F_K) \frac{\omega^n}{n!} + \norm{\Phi^{1/2}(u_\infty)(\bar{\partial} u_\infty)}_{L^p}^2  \\
&\leqslant i \int_{\mathcal{X}} \text{Tr}(u_\infty \Lambda F_K) \frac{\omega^n}{n!} + \norm{\Phi^{1/2}(u_k)(\bar{\partial} u_k)}_{L^p}^2 + 2\epsilon \\
&\leqslant i \int_{\mathcal{X}} \text{Tr}(u_\infty \Lambda F_K) \frac{\omega^n}{n!} + \norm{\Phi^{1/2}(u_k)(\bar{\partial} u_k)}_{L^2}^2 + 3\epsilon \\
&\leqslant  4 \epsilon.
\end{align*}
If a measurable function satisfies an inequality involving the $L^p$-norm uniformly for $p < 2$, then it satisfies the same inequality involving the $L^2$-norm. Because $\epsilon > 0$ was arbitrary, the lemma now follows. \hfill $\Box$

\medskip

We also claim the limit $u_\infty$ is nontrivial. Indeed because the sequence $u_k$ converges to $u_\infty$ in $L^2$, we find that the sequence also converges in $L^4$ by the ideas in the proof of the lemma. But we have $\norm{u_k}_{L^4}^{4} = \norm{|u_k|_K^2}_{L^2}^{2} = 1$. We therefore find that $\norm{u_\infty}_{L^4} = 1$, and $u_\infty$ is nontrivial.

With the lemma, it is possible to see that the eigenvalues of $u_\infty$ are constant. To demonstrate this, we argue that if $\varphi : \mathbb{R} \to \mathbb{R}$ is any smooth function, then the function $\text{Tr}(\varphi(u_\infty))$ is constant. (Here we are using Notation \ref{notation1}.) To prove that this function is constant, we will consider its derivative $\bar{\partial} \text{Tr}(\varphi(u_\infty))$. If $d\varphi : \mathbb{R} \times \mathbb{R} \to \mathbb{R}$ denotes the difference quotient of $\varphi$ as in Lemma \ref{lem:chainrule}, then we have $\bar{\partial} \text{Tr}(\varphi(u_\infty)) = \text{Tr}(d\varphi(u_\infty)(\bar{\partial}u_\infty))$. Let $N$ be a large number. Choose $\Phi : \mathbb{R} \times \mathbb{R} \to \mathbb{R}$ which agrees with $d\varphi$ along the diagonal in the sense that $\Phi(u,u) = d\varphi(u,u) = \varphi'(u)$ and also ensure $\Phi$ satisfies 
\[
N \Phi^2(u,v) < (u-v)^{-1}
\]
for $u < v$. Then by Lemma \ref{lem:chainrule}, we have 
\[
\bar{\partial} \text{Tr}(\varphi(u_\infty)) = \text{Tr}(d\varphi(u_\infty)(\bar{\partial}u_\infty)) = \text{Tr}(\Phi(u_\infty)(\bar{\partial} u_\infty)),
\]
and using Lemma \ref{lem:essential}, we find that 
\[
i  \int_{\mathcal{X}} \text{Tr}(u_\infty \Lambda F_K)\frac{\omega^n}{n!} + N \int_{\mathcal{X}} \langle \Phi^2(u_\infty)(\bar{\partial}u_\infty), \bar{\partial}u_\infty \rangle_K \frac{\omega^n}{n!} \leqslant 0,
\]
that is,
\[
\int_{\mathcal{X}} |\Phi(u_\infty)(\bar{\partial} u_\infty)|_K^2 \frac{\omega^n}{n!} \leqslant -\frac{i}{N} \int_{\mathcal{X}} \text{Tr}(u_\infty \Lambda F_K) \frac{\omega^n}{n!}.
\]
The Schwarz inequality implies that 
\[
\text{Tr}(\Phi(u_\infty)(\bar{\partial}u_\infty)) = \langle \Phi(u_\infty)(\bar{\partial}u_\infty), I_{\mathcal{E}} \rangle_K \leqslant r^2 |\Phi(u_\infty)(\bar{\partial}u_\infty)|_K^2,
\]
from which we obtain 
\[
\norm{\bar{\partial} \text{Tr}(\varphi(u_\infty))}_{L^1} = \norm{\text{Tr}(\Phi(u_\infty)(\bar{\partial}u_\infty))}_{L^1} \leqslant -\frac{r^2i}{N}\int_{\mathcal{X}} \text{Tr}(u_\infty \Lambda F_K) \frac{\omega^n}{n!}.
\] 
We conclude that 
\[
\norm{\bar{\partial} \text{Tr}(\varphi(u_\infty))}_{L^1} \leqslant \frac{C}{N}.
\]
The fact that $N$ was arbitrary implies that $\bar{\partial}\text{Tr}(\varphi(u_\infty)) = 0$. Because the function $\text{Tr}(\varphi(u_\infty))$ is real, we conclude that it must be a constant, as desired. 

If $\nu_1 \leqslant \cdots \leqslant \nu_r$ denote the eigenvalues of $u_\infty$ (which are constant almost everywhere), then we claim that not all $\nu_\alpha$ are equal. Indeed because each $s_k$ satisfies $\int_{\mathcal{X}} \text{Tr}(s_k) \omega^n = 0$, we find also that $\int_{\mathcal{X}} \text{Tr}(u_k)\omega^n = 0$, and hence we have $\int_{\mathcal{X}} \text{Tr}(u_\infty) \omega^n = 0$ as well. But $u_\infty$ is nontrivial, and so at least one eigenvalue must be nonzero. 

It follows that the eigenspaces of $u_\infty$ give rise to a nontrivial flag of $L_1^2$-subbundles of $\mathcal{E}$ which we denote by 
\[
0 \subset \pi_1 \subset \cdots \subset \pi_r = I_\mathcal{E},
\]
where $\pi_\alpha$ denotes projection onto the sum of the first $\alpha$ eigenspaces of $u_\infty$. Note that by construction the $\pi_\alpha$ are self-adjoint with respect to $K$ and satisfy $\pi_\alpha^2 = \pi_\alpha$. 

We claim that each $\pi_\alpha$ represents a weakly holomorphic subbundle of $\mathcal{E}$ in the sense of Definition \ref{def:weak}. For this, it remains only to check that $(I_{\mathcal{E}}-\pi_\alpha)\bar{\partial}\pi_\alpha = 0$. We will use Notation \ref{notation1} to write $\pi_\alpha$ as $\pi_\alpha = p_\alpha(u_\infty)$ where $p_\alpha : \mathbb{R} \to \mathbb{R}$ is a smooth real-valued function satisfying 
\[
p_\alpha(\nu_\beta) = \begin{cases} 
1 & \beta \leqslant \alpha \\
0 &  \beta > \alpha 
\end{cases},
\] 
from which it follows that $\bar{\partial}\pi_\alpha = dp_\alpha(u_\infty)(\bar{\partial}u_\infty)$ by Lemma \ref{lem:chainrule}.

If we set $\Phi_\alpha : \mathbb{R} \times \mathbb{R} \to \mathbb{R}$ to be 
\[
\Phi_\alpha(u,v) = (1-p_\alpha)(v)dp_\alpha(u,v)
\]
where $1$ denotes the constant $1$ function, then 
then we claim that  
\begin{align}\label{eqn:weakPhi}
(I_{\mathcal{E}} - \pi_\alpha)\bar{\partial} \pi_\alpha = \Phi_\alpha(u_\infty)(\bar{\partial}u_\infty).  
\end{align}
Indeed let $e_\beta$ be a unitary basis for $\mathcal{E}$ with respect to which a local expression for $u_\infty$ is 
\[
u_\infty = \sum_{\beta} \nu_\beta e^\beta \otimes e_\beta.
\]
Reasoning in Lemma \ref{lem:chainrule} shows that 
\[
(\bar{\partial}u_\infty)_\beta^\gamma = \begin{cases}
(\nu_\beta - \nu_\gamma)\theta_\beta^\gamma & \beta \neq \gamma \\
0 & \beta = \gamma
\end{cases}
\]
where $\theta_\beta^\gamma$ is the matrix of $\bar{\partial}$. We then compute that the coefficients of $\Phi_\alpha(u_\infty)(\bar{\partial}u_\infty)$ are given by 
\begin{align*}
(\Phi_\alpha(u_\infty)(\bar{\partial}u_\infty))_\beta^\gamma  &= \begin{cases}
(1 - p_\alpha)(\nu_\gamma) dp_\alpha(\nu_\beta, \nu_\gamma)(\nu_\beta - \nu_\gamma)\theta_\beta^\gamma  & \beta \neq \gamma \\
0 & \beta = \gamma
\end{cases} \\
&= \begin{cases}
(1 - p_\alpha(\nu_\gamma)) (p_\alpha(\nu_\beta) - p_\alpha(\nu_\gamma))\theta_\beta^\gamma  & \beta \neq \gamma \\
0 & \beta = \gamma
\end{cases} \\
&= \begin{cases}
(p_\alpha(\nu_\beta) - p_\alpha(\nu_\gamma))\theta_\beta^\gamma  & \beta \neq \gamma, \gamma > \alpha \\
0 & \beta = \gamma \; \text{or} \; \gamma \leqslant \alpha 
\end{cases}.
\end{align*}
On the other hand, the previous paragraph implies that the coefficients of $\bar{\partial}\pi_\alpha$ are given by 
\begin{align*}
(\bar{\partial}\pi_\alpha)_\beta^\mu &= \begin{cases}
(p_\alpha(\nu_\beta) - p_\alpha(\nu_\mu)) \theta_\beta^\mu & \beta \neq \mu \\
0 & \beta = \mu
\end{cases}
\end{align*}
and also the coefficients of $I_\mathcal{E} - \pi_\alpha$ are given by 
\[
(I_\mathcal{E} - \pi_\alpha)_{\mu}^\gamma = \begin{cases}
\delta_\mu^\gamma & \gamma > \alpha \\
0 & \gamma \leqslant \alpha
\end{cases}.
\]
The composition $(I_\mathcal{E} - \pi_\alpha) \bar{\partial}\pi_\alpha$ therefore has coefficients 
\begin{align*}
((I_\mathcal{E} - \pi_\alpha) \bar{\partial}\pi_\alpha)_\beta^\gamma &= (I_\mathcal{E} - \pi_\alpha)_\mu^\gamma (\bar{\partial}\pi_\alpha)_\beta^\mu \\
&= \begin{cases}
(\bar{\partial} \pi_\alpha)_{\beta}^\gamma &\gamma > \alpha \\
0 & \gamma \leqslant \alpha 
\end{cases}.
\end{align*}
Comparing with the coefficients for $\Phi_\alpha(u_\infty)(\bar{\partial}u_\infty)$ we find the relation \eqref{eqn:weakPhi} is indeed true.

We next claim that for $\nu_\gamma > \nu_\beta$, we have $\Phi_\alpha(\nu_\gamma, \nu_\beta) = 0$. There are two possibilities for $\nu_\beta$: either $\nu_\beta \leqslant \nu_\alpha$ or $\nu_\beta > \nu_\alpha$.  If $\nu_\beta \leqslant \nu_\alpha$, then $p_\alpha(\nu_\beta) = 1$ and so $\Phi_\alpha(\nu_\gamma, \nu_\beta) = 0$ by definition. On the other hand, if $\nu_\beta > \nu_\alpha$, then for $\nu_\gamma > \nu_\beta \geqslant \nu_\alpha$, we have that each $p_\alpha(\nu_\gamma) = p_\alpha(\nu_\beta) = 0$, and so the difference quotient $dp(\nu_\gamma, \nu_\beta)$ vanishes. The claim now follows. 

Because the eigenvalues of $u_\infty$ are constant almost everywhere, the construction $\Phi_\alpha(u_\infty)$ depends only on the values of $\Phi_\alpha : \mathbb{R} \times \mathbb{R} \to \mathbb{R}$ on the pairs of eigenvalues $(\nu_\beta, \nu_\gamma)$.  So by replacing $\Phi_\alpha$ with $\Phi_\alpha^N$ satisfying $\Phi_\alpha^N(\nu_\gamma, \nu_\beta) = \Phi_\alpha(\nu_\gamma, \nu_\beta)$ and 
\[
N(\Phi_\alpha^N)^2(u,v) < (u-v)^{-1} \hspace{5mm} \text{for $u > v$},
\]
then we find that still we have 
\[
(I_{\mathcal{E}} - \pi_\alpha)\bar{\partial} \pi_\alpha = \Phi_\alpha^N(u_\infty)(\bar{\partial}u_\infty),
\]
but now we have guaranteed in addition that $\norm{\Phi_\alpha^N(u_\infty)(\bar{\partial} u_\infty)}_{L^2}^2 \leqslant C/N$ by following the line of reasoning from earlier in the argument. Because $N$ is arbitrary, we can conclude that $\Phi_\alpha^N(u_\infty)(\bar{\partial} u_\infty) = 0$. This therefore completes the proof that $\pi_\alpha$ is a weakly holomorphic subbundle. 

We finally show that at least one of the $\pi_\alpha$ for $\alpha < r$ is destabilizing. In a telescoping manner we may write 
\[
u_\infty = \nu_r I_{\mathcal{E}} - \sum_{\alpha = 1}^{r-1} (\nu_{\alpha+1} - \nu_{\alpha}) \pi_\alpha.
\]
Then, according to Definition \ref{def:weak}, the following sum of degrees is given by 
\begin{align*}
W &= \nu_r \deg(E) - \sum_\alpha (\nu_{\alpha+1} - \nu_\alpha) \deg(\pi_\alpha) \\
&= \nu_r \frac{i}{2\pi} \int_{\mathcal{X}} \text{Tr}(\Lambda F_K) - \sum_\alpha (\nu_{\alpha + 1} - \nu_\alpha) \left(\frac{i}{2\pi} \int_{\mathcal{X}} \text{Tr}(\pi_\alpha \Lambda F_K) - \frac{1}{2\pi}\int_{\mathcal{X}} |\bar{\partial} \pi_\alpha|_K^2\right) \\
&= \frac{i}{2\pi} \int_{\mathcal{X}} \text{Tr}(u_\infty \Lambda F_K) + \frac{1}{2\pi}\int_{\mathcal{X}} \sum_\alpha (\nu_{\alpha +1} - \nu_\alpha) |\bar{\partial} \pi_\alpha|_K^2.
\end{align*}
Because $\bar{\partial} \pi_\alpha = p_\alpha(u_\infty)(\bar{\partial} u_\infty)$, we obtain 
\[
W = \frac{i}{2\pi} \int_{\mathcal{X}} \text{Tr}(u_\infty \Lambda F_K) + \frac{1}{2\pi}\int_{\mathcal{X}} \sum_\alpha (\nu_{\alpha +1} - \nu_\alpha) \langle (dp_\alpha)^2(u_\infty)(\bar{\partial} u_\infty), \bar{\partial} u_\infty \rangle_K.
\]
For fixed $\nu_\beta > \nu_\gamma$, if $\nu_\alpha$ satisfies $\nu_\beta > \nu_\alpha \geqslant \nu_\gamma$, then $dp_\alpha(\nu_\beta, \nu_\gamma)^2 = (\nu_\beta - \nu_\gamma)^{-2}$, and vanishes otherwise. It follows that for $\nu_\beta > \nu_\gamma$, the (telescoping) sum satisfies 
\[
\sum_{\alpha} (\nu_{\alpha + 1} - \nu_\alpha) (dp_\alpha)^2(\nu_\gamma, \nu_\beta) = \frac{\nu_\beta - \nu_\gamma}{(\nu_\beta - \nu_\gamma)^{2}} = \frac{1}{\nu_\beta - \nu_\gamma}.
\]
Lemma \ref{lem:essential} implies that $W\leqslant 0$, which means that 
\begin{align}\label{ineq:deg}
\nu_r \deg(\mathcal{E}) \leqslant  \sum_\alpha (\nu_{\alpha + 1} - \nu_\alpha) \deg(\pi_\alpha). 
\end{align}
On the other hand, the trace of $u_\infty$ is zero, which means that 
\[
\nu_r \text{rk}(\mathcal{E}) = \sum_{\alpha}(\nu_{\alpha + 1} - \nu_\alpha) \text{Tr}(\pi_\alpha).
\]
If each $\deg(\pi_\alpha)$ satisfied $\deg(\pi_\alpha) < \text{Tr}(\pi_\alpha) (\deg(\mathcal{E})/\text{rk}(\mathcal{E}))$, then we would have 
\begin{align*}
\sum_{\alpha}(\nu_{\alpha +1} - \nu_\alpha) \deg(\pi_\alpha) < \frac{\deg(\mathcal{E})}{\text{rk}(\mathcal{E})}\sum_\alpha(\nu_{\alpha + 1} - \nu_\alpha) \text{Tr}(\pi_\alpha)  = \nu_r \deg(\mathcal{E}),
\end{align*}
which contradicts \eqref{ineq:deg}. It follows that at least one $\pi_\alpha$ has $\mu(\pi_\alpha) \geqslant \mu(\mathcal{E})$.

This completes the proof of Lemma \ref{lem:proper}. \hfill $\Box$

\newpage

\bibliographystyle{abbrv}
\bibliography{HEmetrics}

\end{document}